\documentclass[a4paper,english,oneside]{amsart}

\usepackage[T1]{fontenc}
\usepackage[english]{babel}

\usepackage{geometry}
\usepackage{tikz,epigraph,graphics,wrapfig}
    \usetikzlibrary{matrix,arrows,decorations.pathmorphing,cd}

\usepackage{mathtools,amsmath,amssymb,verbatim,mathrsfs,hyperref,cleveref,cite,enumitem}
    \crefformat{subsection}{\S#2#1#3}

\swapnumbers
\numberwithin{equation}{subsubsection}

\theoremstyle{plain}                    %
    \newtheorem{thm}[subsubsection]{Theorem}     %
    \newtheorem{prop}[subsubsection]{Proposition}    %
    \newtheorem{corol}[subsubsection]{Corollary}     %
    \newtheorem{lem}[subsubsection]{Lemma}         %

\theoremstyle{definition}               %
    \newtheorem{defin}[subsubsection]{Definition}  %
    \newtheorem{ex}[subsubsection]{Example}    %

\theoremstyle{remark}                   %
    \newtheorem{rem}[subsubsection]{Remark}      %

\newcommand{\R}{\mathbf{R}}     %
\newcommand{\C}{\mathbf{C}}     %
\newcommand{\Z}{\mathbf{Z}}     %
\newcommand{\g}{\mathfrak{g}}   %
\newcommand{\tc}{\mathfrak{t}}  %
\newcommand{\Oo}{\mathcal{O}}   %
\newcommand{\m}{\mathfrak{m}}   %
\newcommand{\T}{\mathscr{T}}    %
\newcommand{\ua}{\underline{a}}   %
\newcommand{\up}{\underline{p}}   %
\newcommand{\bd}{\mathbf{d}}  %
\newcommand{\ud}{\underline{\mathbf{d}}}  %
\newcommand{\uQ}{\underline{Q}}   %
\newcommand{\wm}{\mathcal{WM}_{g, m}^{\tc, \leq \underline{p}}} %
\newcommand{\wmd}{\mathcal{WM}_{g, m}^{\tc, \leq \underline{p}, \ud}} %
\newcommand{\Mgm}{\mathcal{M}_{g, m}} %
\newcommand{\wma}{\mathsf{WM}_{g, m}^{\tc, \leq \underline{p}}} %
\newcommand{\zwma}{\mathsf{WM}_{0, m}^{\tc, \leq \underline{p}}} %
\newcommand{\wmda}{\mathsf{WM}_{g, m}^{\tc, \leq \underline{p}, \ud}} %
\newcommand{\zwmda}{\mathsf{WM}_{0, m}^{\tc, \leq \underline{p}, \ud}} %
\newcommand{\Mgma}{\mathsf{M}_{g, m}} %
\newcommand{\zMgma}{\mathsf{M}_{0, m}} %
\newcommand{\sch}{\mathrm{Sch}_\C}   %
\newcommand{\man}{\mathrm{Man}_\C}   %
\newcommand{\I}{\mathcal{I}}     %
\newcommand{\Spec}{\mathrm{Spec}}     %
\newcommand{\Id}{\mathrm{Id}}   %
\newcommand{\IT}{\mathscr{I \!\! T}}    %
\newcommand{\M}{\mathsf{M}}    %

\title{Moduli spaces of untwisted wild Riemann surfaces}

\author[J.~Douçot]{Jean Douçot}
    \address[J.~Douçot]{``Simion Stoilow'' Institute of mathematics of the Romanian Academy,
Calea Griviței 21, 010702-Bucharest, Sector 1, Romania}
\curraddr{Institut de Recherche Mathématique Avancée,
7 rue René Descartes,
67084 Strasbourg Cedex,
France}
    \email{jean.doucot@math.unistra.fr}
    \thanks{J.D. has been funded by the European PNRR Grant CF 44/14.11.2022, 
    `Cohomological Hall Algebras of Smooth Surfaces and Applications', and more recently by the French DAG-Arts ANR project ANR-24-CE40-4098.}

\author[G. Rembado]{Gabriele Rembado}
    \address[G.~Rembado]{Institut Montpelliérain Alexander Grothendieck,
    University of Montpellier, 
    Place Eugène Bataillon,
    34090 Montpellier,
    France}
    \curraddr{Departament de Matemàtiques,
            Universitat Politècnica de Catalunya (UPC),
Campus Diagonal Sud, 
Av. Doctor Marañón 44-50,
08028 Barcelona,
Spain}
    \email{gabriele.rembado@upc.edu}
    \thanks{G.R. has received funding from the European Union's Horizon Europe research and innovation programme under the Marie Skłodowska-Curie grant agreement No 101108575 (project QuantMod, `Quantisation of moduli spaces: Hitchin connections and isomonodromic deformations');
more recently,
also by the \emph{Ministerio de Ciencia,
Innovación y Universidades},
under the grant PID2024-155686NB-I00.}

\author[M.~Tamiozzo]{Matteo Tamiozzo}
    \address[M.~Tamiozzo]{Université Sorbonne Paris Nord,
    Institut Galilée,
LAGA,
Villetaneuse,
93430 (France)}
    \email{tamiozzo@math.univ-paris13.fr}
    \thanks{M.T. carried out part of this work while being a postdoc at Warwick,
    funded by the ERC grant n.~101001051.}

    \subjclass[2020]{Primary 14D23; Secondary 14H15, 14F35, 57K20, 34M56}
\keywords{Stacks, isomonodromic deformations, mapping class groups}

\begin{document}
\begin{abstract}
	We construct moduli stacks of wild Riemann surfaces in the (pure) untwisted case, for any complex reductive structure group, and we define the corresponding (pure) wild mapping class groups.
\end{abstract}

\maketitle

\setcounter{tocdepth}{1}
\tableofcontents

\section{Introduction}

\subsection{Wild Riemann surfaces}

Wild (families of) Riemann surfaces, introduced by Boalch~\cite{boa14}, are central objects for the intrinsic (topological) study of isomonodromic deformations of irregular singular meromorphic connections.
Classically, for a pointed compact Riemann surface $(\Sigma, \ua)$, where $\ua=(a_1, \ldots, a_m) \in \Sigma^m$ is an ordered $m$-tuple of distinct points, one can look at the character variety $\mathcal M_{Betti}$ parametrising representations of $\pi_1(\Sigma\smallsetminus \ua)$ valued in a complex reductive algebraic group $G$.
When $(\Sigma, \ua)$ varies in a holomorphic family over a connected complex manifold $B$, one obtains a family of character varieties over $B$, and an action of the fundamental group of $B$ on $\mathcal M_{Betti}$.
Fixing the genus $g$ of the Riemann surface and the number $m$ of marked points and considering the resulting moduli stack of pointed Riemann surfaces, one is led to the action of mapping class groups on character varieties, subject of much research (cf.~\cite{gol97},~\cite{wen11},~\cite{pal14},~\cite{gw20},~\cite{glx21} among many others).

Recall that representations $\pi_1(\Sigma\smallsetminus \ua)\rightarrow G$ correspond to $G$-bundles on $\Sigma$ with a meromorphic connection having regular singularities at $\ua$.
In a series of works~\cite{boa01},~\cite{boa02},~\cite{boa07},~\cite{boa12},~\cite{boa14}, Boalch gradually extended the above theory to the more general setting of irregular singularities, building on previous works~\cite{bjl79},~\cite{mr91},~\cite{lr94} (further references and remarks on the history of the subject can be found in the introductions of~\cite{boa21} and \cite{bdr22}).
The starting point was~\cite{boa01}, which studied generic meromorphic connections on vector bundles on the projective line, showed that the Jimbo--Miwa--Ueno isomonodromic deformation equations~\cite{jmu81} arise from flat symplectic Ehresmann connections (irregular
isomonodromy connections), and noted that they are generalisations of the nonabelian Gau\ss–-Manin connections (cf.~the introduction and Theorem 7.1 in~\cite{boa01}).

This led  to the wild generalisation of nonabelian Hodge theory on curves~\cite{sab99}, \cite{bibo04} and in turn provided the first motivation for the definition of wild Riemann surfaces and wild character varieties (in the untwisted case, which is our focus here), given in~\cite{boa14}.

More precisely, in \emph{loc. cit.}, on the one hand Boalch constructed (algebraically) general wild character varieties, generalising the spaces $\mathcal M_{Betti}$ (cf.~\cite[\S~8]{boa14}, extending the algebraic construction in the generic setting of~\cite{boa07} and complementing the analytic constructions of~\cite{boa01},~\cite{bibo04}).
On the other hand, he defined admissible families of wild Riemann surfaces over a base $B$, i.e., pointed families of Riemann surfaces over $B$ together with an \emph{admissible family of irregular types} (whose definition will be recalled below).
The main result of~\cite{boa14} (Theorem 10.2) asserts that, given an admissible family of wild Riemann surfaces over a base $B$, the resulting wild character varieties assemble into a local system of Poisson varieties over $B$, with a complete flat Ehresmann connection - the irregular isomonodromy connection.
In particular, one obtains an action of the fundamental group of $B$ on the relevant wild character variety.
As envisaged in~\cite[\S~3]{boa18} and~\cite[\S~8]{boa14bis}, constructing moduli stacks of (admissible families of) wild Riemann surfaces would enable to define \emph{wild mapping class groups} and study their action on wild character varieties.

The main aim of this document is to construct such stacks and define wild mapping class groups.
We first recall the definitions of irregular types and wild Riemann surfaces, following~\cite[\S~1]{drt26}.

\subsubsection{Irregular types}
\label{irr-types-riem}

Let $\tc \subset \g$ be a Cartan subalgebra of a finite-dimensional complex reductive Lie algebra; fix integers $p \geq 0$ and $g \geq 0$.
Given a compact Riemann surface $\Sigma$ of genus $g$ and a point $a \in \Sigma$ we consider the completed local ring $\hat{\Oo}_{\Sigma, a}$ of $\Sigma$ at $a$, with maximal ideal $\hat{\m}_{\Sigma, a}$.
We denote by $\T_{\Sigma, a}^{\leq p}$ the quotient $\hat{\m}_{\Sigma, a}^{-p}/\hat{\Oo}_{\Sigma, a}$: it is the $\C$-vector space of germs of meromorphic functions at $a$ with pole order at most $p$, up to holomorphic terms.
An (untwisted) irregular type with pole order bounded by $p$ at $a$ is an element $Q \in \tc \otimes_\C \T^{\leq p}_{\Sigma, a}$~\cite[Definition 7.1]{boa14}.

\subsubsection{Families of wild Riemann surfaces}

Let $B$ be a complex manifold, $m \geq 1$ an integer, and $\up=(p_1, \ldots, p_m)\in \Z_{\geq 0}^m$ an $m$-tuple of nonnegative integers.
A $B$-family of wild Riemann surfaces of genus $g$ with $m$ marked points, and with pole orders at most $\up$, is a triple $(\pi \colon \Sigma\rightarrow B, \ua, \uQ)$ where $\pi \colon \Sigma \rightarrow B$ is a holomorphic family of compact Riemann surfaces of genus $g$, $\ua=(a_1, \ldots, a_m)$ is an $m$-tuple of mutually non-intersecting holomorphic sections of $\pi$, and $\uQ=(Q_1, \ldots, Q_m)$ is an $m$-tuple of families of irregular types $Q_i(b) \in \tc \otimes_\C \T^{\leq p_i}_{\pi^{-1}(b), a_i(b)}$, holomorphically varying with $b \in B$ (cf.~\cite[Definition 10.1]{boa14},~\cite[Definition 1.1]{drt26} and Definition~\ref{def-wrs} below).

\subsection{Aim of the text}

For every complex manifold $B$, we consider the groupoid $\wma(B)$ of $B$-families of wild Riemann surfaces with $m$ marked points and pole orders at most $\up$ (morphisms being isomorphisms commuting with the projection to $B$, respecting the sections and the irregular types).
In the body of the text, we will first show that the assignment $B\mapsto \wma(B)$ is an analytic stack, and the natural map $\wma \longrightarrow \Mgma$ to the stack of $m$-pointed genus $g$ compact Riemann surfaces is (representable by) a vector bundle.
Subsequently, we will study the substacks $\wmda$ of $\wma$ obtained fixing the order $\ud=(\mathbf{d}_1, \ldots, \mathbf{d}_m)$ of the poles of the irregular types after evaluation at the roots of $\g$.
The fundamental groups of these substacks will be the (global) wild mapping class groups whose definition is one of the main aims of this text.
Finally, in the last part of the document we will describe more explicitly $\wma$ and $\wmda$ in some examples; we will then define algebraic versions of these stacks, and show that they are Deligne--Mumford if and only if $2g-2+m+\sum_i \max(\mathbf{d}_i)>0$ - as foreseen in \cite[p. 399]{bdr22}-\footnote{Incidentally, as per op.cit., it would be interesting to consider the global twisted/ramified case.}~generalising the analogous result for moduli stacks of curves.

\subsection{General notation and terminology}
\label{gen-not}

The following notation and conventions will be in force throughout the document, unless otherwise stated:

\begin{itemize}
	\item fix a finite-dimensional complex reductive Lie algebra $\g$ and a Cartan subalgebra $\tc$, and let $\Phi$ be the associated root system;

	\item let $r$ be the dimension of $\tc$ as a complex vector space;

	\item fix integers $m \geq 1$ and $g \geq 0$, and a tuple $\up=(p_1, \ldots, p_m) \in \Z_{\geq 0}^m$;

	\item the category of complex manifolds (and holomorphic maps) is denoted by $\man$; the letter $B$ will always denote a complex manifold.

\end{itemize}

Further terminology and background on stacks can be found in appendix~\ref{sec:appendix}.

\subsection{Acknowledgements} We are grateful to the anonymous referee for helpful comments and suggestions that helped improve the exposition.

\section{Wild Riemann surfaces and their moduli}

In this section, we will first recall the definition of families of wild Riemann surfaces over a complex manifold $B$.
Then we will define and study their moduli space.

\subsection{Families of Riemann surfaces and their sections}

Let us start by defining families of (compact) Riemann surfaces of genus $g$ over a complex manifold $B$.
We will call them Riemann surfaces over $B$.

\begin{defin}
	\label{def-an-curve}

	A Riemann surface over $B$ of genus $g$ is a proper holomorphic submersion $\pi \colon \Sigma \rightarrow B$ of complex manifolds, such that all the fibres are Riemann surfaces of genus $g$.
\end{defin}

\subsubsection{Local description of sections}
\label{locsect-anal}

If $\sigma \colon B \rightarrow \Sigma$ is a (holomorphic) section of a Riemann surface over $B$, then the differential of $\sigma$ is injective at every point $b \in B$.
As a consequence of the implicit function theorem, one shows (cf.~\cite[Corollary 1.1.12]{huy05}) that, in suitable local charts around $b$ and $\sigma(b)$, the map $\sigma$ has the form $(z_1, \ldots, z_n)\mapsto (z_1, \ldots, z_n, 0)$.
In particular $\sigma(B)$ is an effective Cartier divisor: the sheaf of $\Oo_{\Sigma}$-ideals $\I_\sigma$ consisting of functions vanishing on $\sigma(B)$ is locally free of rank one.
It sits in a short exact sequence
\begin{equation*}
	0\rightarrow \I_\sigma \rightarrow  \Oo_{\Sigma} \rightarrow \sigma_{*}\Oo_B \rightarrow 0.
\end{equation*}
Furthermore, letting $\Oo_\Sigma(\sigma)=\mathcal{H}om_{\Oo_\Sigma}(\I_\sigma, \Oo_\Sigma)$, the evaluation map $\I_\sigma\otimes_{\Oo_\Sigma}\Oo_\Sigma(\sigma)\rightarrow \Oo_\Sigma$ is an isomorphism.

For an integer $k\geq 0$, let $\I_\sigma^k$ be the $k$-fold product of the ideal sheaf $\I_\sigma$; it is isomorphic to the $k$-fold tensor product of $\I_\sigma$ (over $\Oo_\Sigma$), because the latter is locally free.
The inverse of $\I_\sigma^k$, denoted by $\Oo_\Sigma(k\sigma)$, is isomorphic to the $k$-fold tensor product of $\Oo_\Sigma(\sigma)$.

\subsection{Sheaves of tails of meromorphic functions with bounded pole along a section}
\label{tails}

Fix an integer $k\geq 0$, and a Riemann surface $\pi \colon \Sigma \rightarrow B$ over $B$ with a section $\sigma$. In \cite[\S~10]{boa14} Boalch considers irregular types on $\pi^{-1}(b)$ with a pole at $\sigma(b)$, varying smoothly with $b \in B$.

The aim of this section is to express in sheaf-theoretic terms the notion of ``holomorphic family of irregular types'' on $\Sigma$ with a pole at $\sigma$ of order at most $k$. This will allow us to give a simple proof of Proposition~\ref{irrt-vectbun} below, on which our subsequent arguments are based.

\subsubsection{}

Tensoring the exact sequence $0 \rightarrow \I_\sigma^k \rightarrow \Oo_\Sigma \rightarrow \Oo_\Sigma/\I_\sigma^k \rightarrow 0$ with $\Oo_\Sigma(k\sigma)$ we obtain a short exact sequence
\begin{equation*}
	0 \rightarrow \Oo_\Sigma \rightarrow \Oo_\Sigma(k\sigma) \rightarrow \T_{\Sigma, \sigma}^{\leq k}\rightarrow 0,
\end{equation*}
whose cokernel $\T_{\Sigma, \sigma}^{\leq k}$ will be called the sheaf of tails of meromorphic functions on $\Sigma$ with pole along $\sigma(B)$, of order bounded by $k$. This object is central to our construction: on the one hand, we will use it to define (untwisted) irregular types in \cref{irr-types}; on the other hand, exponential factors of such irregular types are global sections of $\T_{\Sigma, \sigma}^{\leq k}$, cf. \cref{subsec:defwmda}.

\begin{ex}\leavevmode
	\label{tails-ex}

	\begin{enumerate}
		\item Let $\Sigma$ be a compact Riemann surface and let $\sigma \colon \{*\}\rightarrow \Sigma$ be the map with image a point $a \in \Sigma$. Let $\m_{\Sigma, a}$ be the maximal ideal of the local ring $\Oo_{\Sigma, a}$. Then $\T_{\Sigma, a}^{\leq k}$ is the skyscraper sheaf at $a$ with stalk $\m_{\Sigma, a}^{-k}/\Oo_{\Sigma, a}\simeq \hat{\m}_{\Sigma, a}^{-k}/\hat{\Oo}_{\Sigma, a}$.
		\item For a family $\Sigma \rightarrow B$ with a section $\sigma$, choosing charts around a point $b \in B$ and $\sigma(b) \in \Sigma$ as in~\cref{locsect-anal}, meromorphic functions of the form 		$\frac{f}{z_{n+1}^k}$ with $f \colon \C^{n+1} \rightarrow \C$ holomorphic give rise to sections 				of $\T_{\Sigma, \sigma}^{\leq k}$ in a neighbourhood of $\sigma(b)$.
		      In general, sections of $\T_{\Sigma, \sigma}^{\leq k}$ can be written in this form only locally.
	\end{enumerate}
\end{ex}

\subsubsection{Properties of $\T_{\Sigma, \sigma}^{\leq k}$}

For $k=0$ the sheaf $\T_{\Sigma, \sigma}^{\leq k}$ is the zero sheaf; let us now assume that $k\geq 1$. For every $x\in \Sigma \smallsetminus \sigma(B)$ there is an open neighbourhood $U$ of $x$ which does not intersect $\sigma(B)$, hence $\I_{\sigma}|U=\Oo_{\Sigma}|U$.
It follows that $\T_{\Sigma, \sigma}^{\leq k}$ is a sheaf supported on $\sigma(B)$ and annihilated by $\I_\sigma^k$, so it is the pushforward of a unique sheaf, denoted by $\T_{\sigma}^{\leq k}$, on the complex analytic space $(\sigma(B), \Oo_\Sigma/\I_\sigma^k)$.
In particular, even though $\T_{\Sigma, \sigma}^{\leq k}$ was defined as a quotient of a sheaf of meromorphic functions on the whole $\Sigma$, it only carries information about germs of such functions around $\sigma(B)$.
Precisely, if $V \subset U \subset \Sigma$ are two opens having the same intersection with $\sigma(B)$, then $\T_{\Sigma, \sigma}^{\leq k}(U)=\T_{\Sigma, \sigma}^{\leq k}(V)$.

We will denote by $B^{(k)}$ the analytic space $(\sigma(B), \Oo_\Sigma/\I_\sigma^k)$, and by $\pi^{(k)} \colon B^{(k)}\rightarrow B$ the composition of  $\pi$ and the inclusion of $B^{(k)}$ in $\Sigma$.

\begin{lem}
	\label{push-free}

	The pushforward $\pi_*\T_{\Sigma, \sigma}^{\leq k}$ is a locally free sheaf of rank $k$ on $B$.
\end{lem}

\begin{proof}
	The sheaf $\pi_*\T_{\Sigma, \sigma}^{\leq k}$ coincides with the pushforward of $\T_{\sigma}^{\leq k}$ via the map $\pi^{(k)} \colon B^{(k)}\rightarrow B$.
	Choosing local coordinates around a point of $\sigma(B)$ as in~\cref{locsect-anal}, we see that we can write locally $\T_{\Sigma, \sigma}^{\leq k}=z_{n+1}^{-k}\Oo_{\Sigma}/\Oo_{\Sigma}$, hence $\T_{\sigma}^{\leq k}$ is locally free of rank one.

	On the other hand, we claim that $\pi^{(k)}_*\Oo_{B^{(k)}}$ is a locally free $\Oo_B$-module of rank $k$.
	To prove this, we may, and will, replace $B$ by an open neighbourhood of any of its points, and work with the same local coordinates as before.
	The section $\sigma \colon B \rightarrow \Sigma$ factors through $B^{(k)}$, inducing the natural projection $\Oo_{\Sigma}/z_{n+1}^k\Oo_{\Sigma}\rightarrow \Oo_{\Sigma}/z_{n+1}\Oo_{\Sigma}\simeq\Oo_{B}$.
	The map $\pi^{(k)}$ induces on global sections a section of this projection; hence, the claim follows from Lemma~\ref{sect-free} just below.
\end{proof}

\begin{lem}
	\label{sect-free}

	Let $A$ be a commutative ring with unit, and $0 \neq f\in A$ a nonzerodivisor which is not a unit. Let $k\geq 1$ be an integer and $s \colon A/(f)\rightarrow A/(f^k)$ a ring morphism which is a section of the projection $q \colon A/(f^k)\rightarrow A/(f)$.
	The map $s$ endows $A/(f^k)$ with the structure of a free $A/(f)$-module of rank $k$ with basis $1, f, \ldots, f^{k-1}$.
\end{lem}

\begin{proof}
	By induction on $k$; for $k=1$ the claim is clear, so let us suppose that $k\geq 2$.

	Take $a\in A$ and let $\alpha_0=q(a)$.
	Then $q(a-s(\alpha_0))=q(a)-q\circ s \circ q(a)=0$, so we can write $a=s(\alpha_0)+fa_1$ for some $a_1 \in A/(f^k)$.
	Assume that $\alpha_0' \in A/(f), a_1'\in A/(f^k)$ also satisfy $a=s(\alpha_0')+fa_1'$.
	Then $s(\alpha_0-\alpha_0')=f(a_1'-a_1)$; taking the image via $q$ we find that $\alpha_0=\alpha_0'$, hence $f(a_1'-a_1)=0 \in A/(f^k)$.
	As $f$ is not a zero divisor, we deduce that $a_1'-a_1\in (f^{k-1})$.
	Therefore, we have proved that $\alpha_0$ is unique, and $a_1$ is unique in $A/(f^{k-1})$, which implies the statement by induction.
\end{proof}

\subsubsection{Functoriality}
\label{funct-tail}

For $k \geq 0$, let us set $\T_{|B}^{\leq k}=\pi_*\T_{\Sigma, \sigma}^{\leq k}$.
Let $\varphi \colon B'\rightarrow B$ be a holomorphic map, $\pi' \colon \Sigma'\rightarrow B'$ the base change of $\pi$, and $\sigma' \colon B'\rightarrow \Sigma'$ the pullback of $\sigma$.
We have a commutative diagram

\begin{center}
	\begin{tikzcd}
		\Sigma' \arrow[r, "\psi"] \arrow[dd, bend left, "\pi'"] & \Sigma \arrow[dd, bend left, "\pi"]\\
		B'^{(k)} \arrow[u]
		& B^{(k)} \arrow[u] \\
		B' \arrow[u] \arrow[uu, bend left, "\sigma'"] \arrow[r, swap, "\varphi"]
		& \arrow[u] \arrow[uu, bend left, "\sigma"] B
	\end{tikzcd},
\end{center}
and the square with arrows $\sigma', \sigma, \varphi$ and $\psi$ is Cartesian, hence $\sigma'(B')\subset \Sigma'$ is cut out by the ideal sheaf image of $\psi^*\mathcal{I}_\sigma$.
In other words, we have $\sigma'_*\Oo_{B'}=\Oo_{\Sigma'}/\psi^*\mathcal{I}_\sigma=\psi^*\sigma_*\Oo_B$.

\begin{lem}
	\label{funct-lem}

	The following identities hold:

	\begin{enumerate}
		\item $\psi^*\mathcal{I}_{\sigma}=\mathcal{I}_{\sigma'}$;

		\item $\psi^*\Oo_\Sigma(\sigma)=\Oo_{\Sigma'}(\sigma')$;

		\item $\psi^*\T_{\Sigma, \sigma}^{\leq k}=\T_{\Sigma', \sigma'}^{\leq k}$.
	\end{enumerate}
\end{lem}

\begin{proof}\leavevmode
	\begin{enumerate}
		\item Pulling back the exact sequence $0 \rightarrow \I_\sigma \rightarrow \Oo_\Sigma \rightarrow \sigma_*\Oo_B\rightarrow 0$ via $\psi$, and using that $\psi^*\sigma_*\Oo_B=\sigma'_*\Oo_{B'}$, we get an exact sequence
		      \begin{equation*}
			      \psi^*\I_\sigma \rightarrow \Oo_{\Sigma'} \rightarrow \sigma'_*\Oo_{B'}\rightarrow 0.
		      \end{equation*}
		      Therefore we obtain a surjection $\psi^*\I_\sigma\rightarrow \I_{\sigma'}$ of invertible $\Oo_{\Sigma'}$-modules, which is necessarily an isomorphism.

		\item Pulling back the identity $\I_\sigma\otimes_{\Oo_\Sigma}\Oo_\Sigma(\sigma)\simeq \Oo_{\Sigma}$ and using (1) we deduce that $\psi^*(\Oo_\Sigma(\sigma))$ is the inverse of $\I_{\sigma'}$, hence it coincides with $\Oo_{\Sigma'}(\sigma')$.

		\item The $k$-fold tensor product of the isomorphism in (2) yields an isomorphism $\psi^*\Oo_\Sigma(k\sigma)=\Oo_{\Sigma'}(k\sigma')$, which implies the claim by right-exactness of pullback. \qedhere
	\end{enumerate}
\end{proof}

By adjunction we have a natural transformation $\Id\rightarrow \psi_*\psi^*$, inducing $\pi_*\rightarrow \pi_*\psi_*\psi^*$, so $\pi_*\rightarrow \varphi_*\pi'_*\psi^*$ and finally $\varphi^*\pi_*\rightarrow \pi'_*\psi^*$.
Applying these functors to $\T_{\Sigma, \sigma}^{\leq k}$ and using Lemma~\ref{funct-lem} we get a map
\begin{equation}
	\label{eq:bc_onB}
	\varphi^*\T_{|B}^{\leq k}\rightarrow \T_{|B'}^{\leq k}.
\end{equation}

\begin{lem}
	\label{bc-onB}

	The map~\eqref{eq:bc_onB} is an isomorphism.
\end{lem}

\begin{proof}
	By~\cite[\href{https://stacks.math.columbia.edu/tag/01HQ}{Tag 01HQ}]{stacks-project} and Lemma~\ref{funct-lem} there is a unique map $\gamma \colon B'^{(k)}\rightarrow B^{(k)}$ making the top square of the diagram in~\cref{funct-tail} Cartesian.

	Then we claim that the following diagram is Cartesian as well:
	\begin{center}
		\begin{tikzcd}
			B'^{(k)} \arrow[d, swap, "\pi'^{(k)}"] \arrow[r, "\gamma"] & B^{(k)} \arrow[d, "\pi^{(k)}"] \\
			B' \arrow[r, swap, "\varphi"] & B
		\end{tikzcd}.
	\end{center}
	Indeed
	\begin{equation*}
		B^{(k)}\times_B B'= B^{(k)}\times_\Sigma(\Sigma\times_B B')=B^{(k)}\times_\Sigma \Sigma'=B'^{(k)}.
	\end{equation*}

	Now $\T_{\sigma}^{\leq k}$ is the pullback of $\T_{\Sigma, \sigma}^{\leq k}$ (cf.~the proof of~\cite[\href{https://stacks.math.columbia.edu/tag/08KS}{Tag 08KS}]{stacks-project}), hence using Lemma~\ref{funct-lem} we deduce that $\gamma^* \T_{\sigma}^{\leq k}=\T_{\sigma'}^{\leq k}$.
	The lemma follows from the natural isomorphism $\varphi^*\pi^{(k)}_*(\T_\sigma^{\leq k})\simeq \pi'^{(k)}_*\gamma^*(\T_\sigma^{\leq k})$, which is the base change result~\cite[Theorem 3.4, p. 116]{basta76}.
\end{proof}

\subsection{Irregular types}
\label{irr-types}

Let $\pi \colon \Sigma \rightarrow B$ be a Riemann surface over $B$ with a section $\sigma$, and $k \geq 0$ an integer.
We define an \emph{irregular type} (on $B$) with pole at $\sigma$ of order at most $k$ to be a global section of the sheaf $\tc\otimes_\C \T_{\Sigma, \sigma}^{\leq k}$.

\subsubsection{Irregular type of a connection on a framed bundle}

Assume in this subsection that $B$ is a point, and denote the image of $\sigma$ by $a$. Let $K$ be the fraction field of the completion of the local ring $\mathcal{O}_{\Sigma, a}$, with its valuation $v: K \rightarrow \Z \cup \{+\infty\}$. For $l \in \Z$, let $K_{\geq l}$ be the set of elements $x \in K$ such that $v(x) \geq l$. We will denote the valuation ring $K_{\geq 0}$ by $\mathcal{O}_K$, and the maximal ideal $K_{\geq 1}$ of $\mathcal{O}_K$ by $\mathfrak{m}$. With this notation, an irregular type with pole at $\sigma$ of order at most $k$ is an element of $\tc \otimes_\C (K_{\geq -k}/\mathcal{O}_K)$. Let us recall how, for $\mathfrak{g}=\mathfrak{gl}_r$ and $\mathfrak{t}$ the Cartan algebra of diagonal matrices, one can attach to certain ``framed'' vector bundles with meromorphic connection on $\Sigma$ a unique irregular type. This construction is not used elsewhere in this document, but we include it because it is the main motivation for the definition of irregular type. See \cite[\S 2]{boa02} for a similar discussion in the generic case, for general $\mathfrak{g}$.

Let $\mathcal{V}$ be a vector bundle of rank $r$ on $\Sigma$, and $\nabla$ a meromorphic connection on $\mathcal{V}$ with a pole at $a$ of order at most $k+1$. Looking at the stalk at $a$, and further pulling back to the completed local ring, we obtain a free $\mathcal{O}_K$-module $V$ of rank $r$ together with a meromorphic connection, abusively still denoted by the same symbol,
\begin{equation*}
	\nabla: V \rightarrow \hat{\Omega}^1_{K} \otimes_{\mathcal{O}_K} V,
\end{equation*}
where $\hat{\Omega}^1_{K}$ is the module of continuous Kähler differentials of $K/\C$. Choosing a basis $(e_1, \ldots, e_r)$ of $V$ we obtain an isomorphism of $\Oo_K$-modules $V \simeq \Oo_K^r$, hence there exists a unique $r \times r$ matrix $M$ with coefficients $\omega_{ij} \in \hat{\Omega}^1_{K}$ such that $\nabla(e_i)=\sum_{j=1}^r \omega_{ij} \otimes e_j$ for $1 \leq i \leq r$. Changing basis - replacing the column matrix $E$ with entries $e_1, \ldots, e_r$ by $g E$ for $g \in \mathrm{GL}_r(\Oo_K)$ - results in replacing $M$ by $g M g^{-1}+\mathrm{d} g g^{-1}$.

Let $\hat{\Omega}^1_{\Oo_K}$ be the module of continuous Kähler differentials of $\Oo_K/\C$, $\hat{\Omega}^1_{K, \geq -1}=\hat{\Omega}^1_{\Oo_K}\otimes_{\Oo_K} K_{\geq -1}$ and $\hat{\Omega}^1_{K, <-1}=\hat{\Omega}^1_{K}/\hat{\Omega}^1_{K, \geq -1}$. For $\omega \in \hat{\Omega}^1_{K}$, we denote by $\bar{\omega}$ its image in $\hat{\Omega}^1_{K, <-1}$; let $\bar{M}$ be the matrix with entries $\bar{\omega}_{ij}$. We say that $\nabla$ is \emph{untwisted} if there exists a basis of $V$ such that $\bar{M}$ is diagonal (for a general meromorphic connection, one needs to replace $K$ by a finite extension in order to obtain a diagonal $\bar{M}$, see \cite[Theorem 1]{turr55}). We will explain how to attach to such a $\nabla$ an element of $\tc \otimes (K_{\geq -k}/\mathcal{O}_K)$ which only depends on the reduction of $(e_1, \ldots, e_r)$ modulo $\mathfrak{m}$, i.e., on the choice of a suitable basis of the fibre of $\mathcal{V}$ at $a$.

First of all, note that the derivation $d\colon K \rightarrow \hat{\Omega}^1_K$ induces an isomorphism $K/K_{\geq 0}\xrightarrow{\sim} \hat{\Omega}^1_{K, <-1}$. This can be checked after fixing an isomorphism $\Oo_K \simeq \C[[z]]$ - yielding an isomorphism $\hat{\Omega}^1_K  \simeq \C((z)) \mathrm{d}z$ - and using that $d(z^m)=mz^{m-1}\mathrm{d}z$ for $m \in \Z$. Therefore, given an untwisted connection $\nabla$ with a pole at $a$ of order at most $k+1$ and a basis $(e_1, \ldots, e_r)$ of $V$ such that $\bar{M}$ is diagonal, we obtain a unique element $IT(\nabla, (e_1, \ldots, e_r)) \in \tc \otimes (K_{\geq -k}/\mathcal{O}_K)$ mapping to $\bar{M} \in \tc \otimes_\C \hat{\Omega}^1_{K, <-1}$. It remains to show that $IT(\nabla, (e_1, \ldots, e_r))$ only depends on the reduction of $(e_1, \ldots, e_r)$ modulo $\mathfrak{m}$. In other words, if $g \in \mathrm{GL}_r(\Oo_K)$ projects to the identity in $\mathrm{GL}_r(\C)$ and the image $\bar{M}'$ of $M'=gMg^{-1}$ in $\mathfrak{g}\otimes_\C \hat{\Omega}^1_{K, <-1}$ is diagonal, we need to prove that $\bar{M}'=\bar{M}$.

To show this, let us fix an isomorphism $\Oo_K \simeq \C[[z]]$ as above. This allows us to see $M$ and $M'$ as matrices with coefficients in $\frac{1}{z^{k+1}}\C[[z]]$, and $g$ as an element of $GL_r(\C[[z]])$ congruent to the identity modulo $(z)$. Multiplying $M$ and $M'$ by $z^{k+1}\mathrm{Id}$ turns them into matrices $N$ and $N'$ with $\C[[z]]$-coefficients. Let $\bar{N}$ and $\bar{N}'$ be the images of $N$ and $N'$ in $A=\C[[z]]/(z^k)$; we need to show that $\bar{N}=\bar{N}'$. This follows for instance from \cite[Lemma 8.1.2]{cfrw}, which proves a more general statement, valid for arbitrary $\mathfrak{g}$; let us give a simple argument in the case of $\mathfrak{gl}_r$. Let $(v_1, \ldots, v_r)$ be the canonical basis of $A^r$. By assumption, there exist $\lambda_1, \ldots, \lambda_r \in A$ such that $\bar{N}v_i =\lambda_i v_i$ for $1 \leq i \leq r$, hence $\bar{N}'(gv_i)=\lambda_i\cdot gv_i$. On the other hand, we are assuming that $\bar{N}' v_i = \gamma_i v_i$ for some $\gamma_1, \ldots, \gamma_r \in A$. Writing
\begin{equation*}
	g v_i=\sum_{j=1}^r a_{i, j} v_j \text{ with } a_{i, j} \in A
\end{equation*}
we obtain $\bar{N}'(gv_i)=\sum_{j=1}^ra_{i, j}\gamma_j v_j=\sum_{j=1}^ra_{i, j}\lambda_i v_j$. Finally, the vectors $g(v_i)$ and $v_i$ have the same image in $(A/(z))^r$, hence $a_{i, i}$ belongs to $A^\times$ and $\gamma_i=\lambda_i$.

\begin{rem}
	Let us spell out the above definition of irregular type for $k \geq 1$.

	Recall that $\T_{\Sigma, \sigma}^{\leq k}=\Oo_\Sigma(k\sigma)/\Oo_\Sigma$.
	If $(e_1, \ldots, e_r)$ is a basis of $\tc$ as a complex vector space, then every global section $Q$ of $\tc\otimes_\C \T_{\Sigma, \sigma}^{\leq k}$ can be written uniquely as $\sum_{j=1}^re_j \otimes Q^{(j)}$, with $Q^{(j)}$ a section of $\Oo_\Sigma(k\sigma)/\Oo_\Sigma$.
	Such a section is a meromorphic function defined locally (around each point of $\sigma(B)$, cf.~\cref{tails-ex}) up to holomorphic terms, and with pole along $\sigma(B)$ of order at most $k$.
	Therefore, irregular types can be thought of (locally) as meromorphic $\tc$-valued functions with pole along $\sigma(B)$ of order at most $k$, up to holomorphic terms.
\end{rem}

\subsubsection{}

If $B' \rightarrow B$ is a holomorphic map then, by adjunction and Lemma~\ref{funct-lem}(3), we obtain a map
\begin{equation*}
	\T_{\Sigma, \sigma}^{\leq k} \rightarrow \psi_*\psi^*\T_{\Sigma, \sigma}^{\leq k} \rightarrow \psi_*\T_{\Sigma', \sigma'}^{\leq k},
\end{equation*}
inducing a map on global sections $\Gamma(\Sigma, \tc\otimes_\C \T_{\Sigma, \sigma}^{\leq k})\rightarrow \Gamma(\Sigma', \tc\otimes_\C \T_{\Sigma', \sigma'}^{\leq k})$.
Similarly, for a map $B'' \rightarrow B'$ of manifolds over $B$ we have a map $\Gamma(\Sigma', \tc\otimes_\C \T_{\Sigma', \sigma'}^{\leq k})\rightarrow \Gamma(\Sigma'', \tc\otimes_\C \T_{\Sigma'', \sigma''}^{\leq k})$.
We obtain a contravariant functor $\IT^{\tc, \leq k}_{B, \sigma}$ from complex manifolds over $B$ to sets sending $B'\rightarrow B$ to $\Gamma(\Sigma', \tc\otimes_\C \T_{\Sigma', \sigma'}^{\leq k})$.

\begin{prop}
	\label{irrt-vectbun}

	Let $\pi \colon \Sigma\rightarrow B$ be a Riemann surface over $B$ of genus $g$, $\sigma \colon B \rightarrow \Sigma$ a section and $k \geq 0$ an integer.
	The functor $\IT^{\tc, \leq k}_{B, \sigma}$ is representable by a vector bundle (of rank $rk$) on $B$.
\end{prop}

\begin{proof}
	For $\varphi \colon B' \rightarrow B$, we have
	\begin{equation*}
		\Gamma(\Sigma', \tc\otimes_\C \T_{\Sigma', \sigma'}^{\leq k})=\Gamma(B', \tc\otimes_\C \T_{|B'}^{\leq k})=\Gamma(B', \tc\otimes_\C \varphi^*\T_{|B}^{\leq k}),
	\end{equation*}
	where the last identification follows from Lemma~\ref{bc-onB}.
	Hence, the functor $\IT^{\tc, \leq k}_{B, \sigma}$ is the functor of global sections of $\tc\otimes_\C \T_{|B}^{\leq k}$, which is a locally free sheaf of rank $rk$ over $B$ by Lemma~\ref{push-free}.
	The vector bundle over $B$ attached to this locally free sheaf represents $\IT^{\tc, \leq k}_{B, \sigma}$ (see~\cref{funp-vectbun}).
\end{proof}

\subsubsection{}

We call the vector bundle of Proposition~\ref{irrt-vectbun} the bundle of irregular types on $B$ with pole at $\sigma$ of order at most $k$, and we still denote it by $\IT^{\tc, \leq k}_{B, \sigma}$.
We denote by $\T^{\leq k}_{B, \sigma}$ the vector bundle of rank $k$ over $B$ representing the functor sending $B'\rightarrow B$ to $\Gamma(\Sigma', \T_{\Sigma', \sigma'}^{\leq k})$.

If $\underline{\sigma}=(\sigma_1, \ldots, \sigma_m)$ is an $m$-tuple of non-intersecting sections of $\Sigma \rightarrow B$ and $\underline{k}=(k_1, \ldots, k_m) \in \Z_{\geq 0}^m$, we denote by $\IT^{\tc, \leq \underline{k}}_{B, \underline{\sigma}}$ the fibre product of the bundles $\IT^{\tc, \leq k_i}_{B, \sigma_i}$ over $B$, for $1 \leq i \leq m$.

\subsubsection{The root stratification}
\label{rootstrat}

Given $\pi \colon \Sigma \rightarrow B$ with a section $\sigma$, and an integer $k \geq 0$, we will stratify the space $\IT^{\tc, \leq k}_{B, \sigma}$ using the following inputs:

\begin{enumerate}
	\item for $0 \leq j \leq k$ we have natural injections $\T_{\Sigma', \sigma'}^{\leq j}\rightarrow \T_{\Sigma', \sigma'}^{\leq k}$ for every $B' \rightarrow B$, hence we obtain a morphism $\T^{\leq j}_{B, \sigma}\rightarrow \T^{\leq k}_{B, \sigma}$ identifying the source with a vector sub-bundle of the target;

	\item similarly, we have a morphism $\IT^{\tc, \leq j}_{B, \sigma}\rightarrow \IT^{\tc, \leq k}_{B, \sigma}$;

	\item a root $\alpha \in \Phi$ induces natural surjections of sheaves $\tc\otimes \T_{\Sigma', \sigma'}^{\leq k}\rightarrow \T_{\Sigma', \sigma'}^{\leq k}$ for every $B' \rightarrow B$; this yields a submersion $\IT^{\tc, \leq k}_{B, \sigma}\rightarrow \T^{\leq k}_{B, \sigma}$ of vector bundles over $B$.
\end{enumerate}

Given a collection of integers $\bd=(d_\alpha)_{\alpha \in \Phi}$ with $0 \leq d_\alpha \leq k$ for every $\alpha \in \Phi$, we denote by $\IT^{\tc, \leq k, \leq \bd}_{B, \sigma}$ the intersection of the fibre products $\IT^{\tc, \leq k}_{B, \sigma}\times_{\T^{\leq k}_{B, \sigma}} \T^{\leq d_\alpha}_{B, \sigma} \subset \IT^{\tc, \leq k}_{B, \sigma}$ as $\alpha \in \Phi$ varies.
This is a closed submanifold of $ \IT^{\tc, \leq k}_{B, \sigma}$ (we will justify in~\cref{locstr-str} that it is a manifold); by construction, it is the moduli space of irregular types with pole at $\sigma$ of order bounded by $k$, and such that after evaluation at each $\alpha \in \Phi$ the pole order is bounded by $d_\alpha$.

We write $\bd'=(d'_\alpha)_{\alpha \in \Phi} < \bd=(d_\alpha)_{\alpha \in \Phi}$ if $d'_{\alpha}\leq d_\alpha$ for every $\alpha \in \Phi$, and at least one of the inequalities is strict.
We define
\begin{equation*}
	\IT^{\tc, \leq k, \bd}_{B, \sigma}= \IT^{\tc, \leq k, \leq \bd}_{B, \sigma} \smallsetminus \cup_{\bd'<\bd} \IT^{\tc, \leq k, \leq \bd'}_{B, \sigma} \subset \IT^{\tc, \leq k}_{B, \sigma}.
\end{equation*}
By construction, $\IT^{\tc, \leq k, \bd}_{B, \sigma}$ is an open submanifold of $\IT^{\tc, \leq k, \leq \bd}_{B, \sigma}$, and a locally closed submanifold of $\IT^{\tc, \leq k}_{B, \sigma}$.
It parametrises irregular types with pole at $\sigma$ of order at most $k$, and such that after evaluation at $\alpha \in \Phi$ the pole order is (everywhere) equal to $d_\alpha$.

\subsubsection{Local structure of root strata}
\label{locstr-str}

Let us describe $\IT^{\tc, \leq k, \bd}_{B, \sigma}\rightarrow B$ locally on $B$.

Given $b \in B$, we choose small enough opens $U \ni b$ and $V \ni \sigma(b)$ with charts as in~\cref{locsect-anal}.
Restricted to $U$, the sheaf $\T^{\leq k}_{|B}$ is free over $\Oo_U$, with basis $(z_{n+1}^{-1}, \ldots, z_{n+1}^{-k})$; let us use this basis to identify $\T^{\leq k}_{|B}$ restricted to $U$ with $\Oo_U^k$.
Concretely, the point is that a meromorphic function on $V$ with pole of order at most $k$ along $\sigma(B)$ can be written uniquely, up to holomorphic terms, as a sum $f_1z_{n+1}^{-1}+ \ldots + f_kz_{n+1}^{-k}$, where $f_1, \ldots, f_k$ are holomorphic functions of the variables $z_1, \ldots, z_n$, which are identified with functions on $U$.

Thanks to the above identification, we see that the functor of points of the pullback of $\IT^{\tc, \leq k}_{B, \sigma}$ to $U$ is representable by $\tc^k \times U$.
The pullback of $\IT^{\tc, \leq k, \leq \bd}_{B, \sigma}$ to $U$ is representable by the intersection of the manifolds
\begin{equation*}
	\left(\prod_{1 \leq j \leq d_\alpha}\tc \times \prod_{d_\alpha < j \leq k}\ker(\alpha)\right)\times U, \qquad \alpha \in \Phi.
\end{equation*}
Explicitly, take a meromorphic $\tc$-valued function $f_1z_{n+1}^{-1}+ \ldots + f_kz_{n+1}^{-k}$ on $V$ with pole order at most $k$ along $\sigma(B)$. Applying $\alpha \in \Phi$ to it, as in point (3) of \cref{rootstrat}, we obtain the function $\alpha\circ f_1z_{n+1}^{-1}+ \ldots + \alpha \circ f_kz_{n+1}^{-k}$. The latter has pole order at most $d_\alpha$ if and only if $f_j$ takes values in $\ker(\alpha)$ for $j>d_\alpha$.

In other words, we can write the pullback to $U$ of $\IT^{\tc, \leq k, \leq \bd}_{B, \sigma}$ as the product of $U$ and
\begin{equation*}
	\prod_{j=1}^k\left(\cap_{d_\alpha <j} \ker(\alpha)\right) \subset \tc^k.
\end{equation*}

Each of the spaces in the above product is an intersection of hyperplanes through the origin in $\tc$, cut out by the equations $\alpha=0$ for $\alpha$ such that $d_\alpha<j$.
Hence $\IT^{\tc, \leq k, \leq \bd}_{B, \sigma}$ is a manifold, and the same is true for the open $\IT^{\tc, \leq k, \bd}_{B, \sigma}\subset \IT^{\tc, \leq k, \leq \bd}_{B, \sigma}$. Explicitly, the pullback to an open $U$ as above of $\IT^{\tc, \leq k, \bd}_{B, \sigma}$ is of the form $U\times \prod_{j=1}^{k} C_j$, where, for $1 \leq j \leq k$,
\begin{equation}\label{eq-cj}
	C_j= \bigcap_{d_{\alpha}<j} \mathrm{ker}(\alpha) \cap \bigcap_{d_{\alpha}=j}(\tc \smallsetminus \mathrm{ker}(\alpha))\subset \tc.
\end{equation}

\subsection{Stacks of wild Riemann surfaces}

In this section we will define stacks of (labelled, untwisted) wild Riemann surfaces, of fixed genus, number of marked points, and maximal pole orders of the irregular types.
We start by introducing (families of) wild Riemann surfaces.
Later on we will ask, as in~\cite[Definition 10.1]{boa14}, such families to be \emph{admissible}; this will lead to a stratification of the stack of wild Riemann surfaces, mirroring the stratification on $\IT^{\tc, \leq k}_{B, \sigma}$ described above.

Recall that we fixed an integer $m \geq 1$ and an $m$-tuple $\up=(p_1, \ldots, p_m) \in \Z_{\geq 0}^m$.

\begin{defin}
	\label{def-wrs}

	A wild Riemann surface over a complex manifold $B$ of genus $g$, with $m$ marked points and pole orders bounded by $\up$, is a triple $(\pi \colon \Sigma \rightarrow B, \ua, \uQ)$ consisting of:
	\begin{enumerate}
		\item a Riemann surface $\pi \colon \Sigma \rightarrow B$ of genus $g$ in the sense of Definition~\ref{def-an-curve};

		\item an $m$-tuple $\ua=(a_1, \ldots, a_m)$ of mutually non-intersecting holomorphic sections $a_i \colon B \rightarrow \Sigma$ of $\pi$;

		\item an $m$-tuple $\uQ=(Q_1, \ldots, Q_m)$ of irregular types $Q_i \in \Gamma(\Sigma, \tc\otimes_\C \T_{\Sigma, a_i}^{\leq p_i})$.
	\end{enumerate}
\end{defin}

In particular, note that we have $Q_i=0$ whenever $p_i=0$.

\subsubsection{Definition of $\wma$}

For every complex manifold $B$, the groupoid $\wma(B)$ is the category whose objects are wild Riemann surfaces over $B$ of genus $g$, with $m$ marked points and pole orders bounded by $\up$.
An isomorphism from $(\pi_1 \colon \Sigma_1 \rightarrow B, \ua_1, \uQ_1)$ to $(\pi_2 \colon \Sigma_2 \rightarrow B, \ua_2, \uQ_2)$ is an isomorphism  $\psi \colon \Sigma_1 \rightarrow \Sigma_2$ such that $\pi_2\circ \psi=\pi_1$, the pullback of $\ua_2$ is $\ua_1$ (in other words, $\psi \circ a_{1, i}=a_{2, i}$ for $1 \leq i \leq m$) and the pullback of $\uQ_2$ is $\uQ_1$.

As above we see that any holomorphic map $B'\rightarrow B$ induces a functor from $\wma(B)$ to $\wma(B')$.
Moreover, forgetting irregular types we get functorial maps $\wma(B)\rightarrow \Mgma(B)$ for every manifold $B$, where $\Mgma(B)$ is the groupoid of genus $g$ Riemann surfaces over $B$ with $m$ mutually disjoint sections.

\subsubsection{Definition of $\wmda$}\label{subsec:defwmda}

Let us now define admissible families of wild Riemann surfaces over a manifold $B$.
Fix integers $0 \leq d_{\alpha, i} \leq p_i$ for each $\alpha \in \Phi$ and $1 \leq i \leq m$, and let $\ud=(\bd_1, \ldots, \bd_m)$, where $\bd_i=(d_{\alpha, i})_{\alpha \in \Phi}$.

If $(\pi \colon \Sigma \rightarrow B, \ua, \uQ)$ is a wild Riemann surface as above, then for every $\alpha \in \Phi$ and $1 \leq i \leq m$ we can look at $\alpha \circ Q_i\in \Gamma(\Sigma, \T_{\Sigma, a_i}^{\leq p_i})$ - usually called an exponential factor.
Recall that we have injections of sheaves $\T_{\Sigma, a_i}^{\leq d_{\alpha, i}}\rightarrow \T_{\Sigma, a_i}^{\leq p_i}$ (cf.~\cref{rootstrat}).
We will use the following terminology:

\begin{itemize}
	\item we say that $\alpha \circ Q_i$ has order at most $d_{\alpha, i}$ if it belongs to $\Gamma(\Sigma, \T_{\Sigma, a_i}^{\leq d_{\alpha, i}})$;

	\item we say that $\alpha \circ Q_i$ has order $d_{\alpha, i}$ if it belongs to $\Gamma(\Sigma, \T_{\Sigma, a_i}^{\leq d_{\alpha, i}})$ and, for every $b \in B$, its pullback to the fibre $\Sigma_b=\pi^{-1}(b)$ does not belong to $\Gamma(\Sigma_b, \T_{\Sigma_b, a_{i}(b)}^{\leq d_{\alpha, i}-1})$;

	\item we say that $Q_i$ has root order at most $\bd_i$ (resp. root order $\bd_i$) if $\alpha \circ Q_i$ has order at most (resp. equal to) $d_{\alpha, i}$ for every $\alpha \in \Phi$;

	\item we say that $\uQ$ has root order at most $\ud$ (resp. root order $\ud$) if $Q_i$ has root order at most $\bd_i$ (resp. root order $\bd_i$) for all $1\leq i \leq m$;

	\item we say that $(\pi \colon \Sigma \rightarrow B, \ua, \uQ)$ is admissible of root order $\ud$ if $\uQ$ has root order $\ud$.
\end{itemize}

The groupoid of admissible wild Riemann surfaces (over $B$, of genus $g$, with $m$ marked points and pole orders bounded by $\up$) of root order $\ud$ is denoted by $\wmda(B)$.

\begin{rem}
	Let us rephrase more concretely the condition that $\alpha \circ Q_i$ has order $d_{\alpha, i}$.

	For every $b\in B$, pulling back $\alpha \circ Q_i\in \Gamma(\Sigma, \T_{\Sigma, a_i}^{\leq p_i})$ to the fibre $\Sigma_b$ we obtain an element $(\alpha \circ Q_i)_b \in \Gamma(\Sigma_b, \T_{\Sigma_b, a_{i}(b)}^{\leq p_i})\simeq z^{-p_i}\mathcal{O}_{\Sigma_b, a_{i}(b)}/\mathcal{O}_{\Sigma_b, a_{i}(b)}$, where $z$ is a local coordinate on $\Sigma_b$ around $a_{i}(b)$. Then $\alpha \circ Q_i$ has order $d_{\alpha, i}$ if and only if for every $b \in B$ the pole order of $(\alpha \circ Q_i)_b$ is $d_{\alpha, i}$.
\end{rem}

\begin{prop}
	\label{wmstack}

	The assignment $\wma \colon \man \rightarrow \mathrm{Groupoids}$ sending $B$ to $\wma(B)$ is a stack.
\end{prop}

\begin{proof}
	Let us verify that the properties defining a stack (cf.~\cref{stacks}) are satisfied.

	\begin{description}
		\item[Objects glue] Let $(B_j)_{j \in J}$ be an open cover of a manifold $B$.
		      Suppose we are given objects $(\pi_j \colon \Sigma_j \rightarrow B_j, \ua_j, \uQ_j)$ of $\wma(B_j)$ for each $j$, together with isomorphisms of their restrictions to $B_j \cap B_{j'}$ satisfying the cocycle condition on triple intersections.
		      We need to show that there is $(\pi \colon \Sigma \rightarrow B, \ua, \uQ)$ pulling back to $(\pi_j \colon \Sigma_j \rightarrow B_j, \ua_j, \uQ_j)$ over each $B_j$.
		      Since $\Mgma$ is a stack, we do get $(\pi \colon \Sigma \rightarrow B, \ua)$ with the desired property.
		      Now, each $\uQ_j$ is a collection of sections over $B_j$ of $\tc\otimes_\C(\pi_*\T_{\Sigma, a_{i}}^{\leq p_i}),$ for $1 \leq i \leq m$. Such sections glue to a section on $B$ because $\pi_*\T_{\Sigma, a_{i}}^{\leq p_i}$ is a sheaf on $B$.

		\item[Isomorphisms glue] Isomorphisms of objects in $\wma(B)$ are by definition isomorphisms of Riemann surfaces over $B$ (respecting the extra structures), which glue. \qedhere
	\end{description}
\end{proof}

We can now state the main result of this section. See Definition \ref{def:anstack} for the definition of analytic stack.

\begin{thm}
	\label{mainthm-an}

	Given integers $g\geq 0, m\geq 1$, and tuples
	\begin{equation*}
		\up=(p_1, \ldots, p_m) \in \Z_{\geq 0}^m, \quad \ud=(d_{\alpha, i})_{\alpha \in \Phi, 1\leq i \leq m}, \qquad \text{ with } 0 \leq d_{\alpha,i}\leq p_i,                                                                                                                                                  \end{equation*}
	the following assertions hold true.

	\begin{enumerate}
		\item The stack $\wma$ is analytic.

		\item The map $\wma\rightarrow \Mgma$ is representable by vector bundles.
		      More precisely, let $B$ be a complex manifold and $B\rightarrow \Mgma$ a map corresponding to an $m$-pointed genus $g$ Riemann surface $(\Sigma\rightarrow B, \ua)$, where $\ua=(a_1, \ldots, a_m)$.
		      The fibre product $\wma \times_{\Mgma}B$ is isomorphic to $\IT^{\tc, \leq \up}_{B, \ua}$.

		\item The assignment  $\wmda \colon \man \rightarrow \mathrm{Groupoids}$ sending $B$ to $\wmda(B)$ is an analytic stack.

		\item The map $\wmda\rightarrow \Mgma$ is representable by manifolds which, if non-empty, are locally products of hyperplane complements in affine spaces (in particular, it is a submersion).

		\item The fibre product of  $\wmda$ with a point over $\Mgma$, if non-empty, is an $m$-fold product of products of hyperplane complements as in~\cite[Definition 2.3]{drt26}.
	\end{enumerate}
\end{thm}

\begin{proof}
	We know from the above proposition that $\wma$ is a stack, which implies that the same is true for $\wmda$.

	Let us now prove $(2)$.
	Let $B\rightarrow \Mgma$ be a manifold as in the statement, and $B'$ another manifold.
	By definition (cf.~\cref{atlas}), the groupoid $\wma \times_{\Mgma}B(B')$ has the following description:

	\begin{itemize}
		\item objects are triples $(\varphi \colon B'\rightarrow B, (\pi' \colon \Sigma'\rightarrow B', \ua', \uQ'), \tau)$, where $\tau \colon \Sigma\times_{B, \varphi} B'\rightarrow \Sigma'$ is an isomorphism commuting with the maps to $B'$ and respecting marked points;

		\item isomorphisms between $(\varphi \colon B'\rightarrow B, (\pi'_1 \colon \Sigma'_1\rightarrow B', \ua'_1, \uQ'_1), \tau_1)$ and $(\varphi \colon B'\rightarrow B, (\pi_2' \colon \Sigma'_2\rightarrow B', \ua'_2, \uQ'_2), \tau_2)$ are isomorphisms $\psi \colon \Sigma'_1\rightarrow \Sigma'_2$ commuting with the structure maps to $B'$, respecting marked points and irregular types, and such that $\psi \circ \tau_1=\tau_2$.
	\end{itemize}

	Therefore, we see that every object of $\wma \times_{\Mgma}B(B')$ is isomorphic to one of the form $(\varphi \colon B'\rightarrow B, (\Sigma\times_B B'\rightarrow B', \varphi^*(\ua), \uQ'), \mathrm{Id})$, where $\uQ'$ is a collection of irregular types on $(\Sigma\times_B B', \varphi^*(\ua))$; furthermore, objects of this form have no automorphisms.
	In other words, the groupoid $\wma \times_{\Mgma}B(B')$ is equivalent to $\IT^{\tc, \leq \up}_{B, \ua}(B')$.

	Let us now deduce $(1)$, i.e., that $\wma$ is an analytic stack.
	Choose a surjective submersion $U\rightarrow \Mgma$ from a complex manifold $U$.
	Then the fibre product $V=U\times_{\Mgma} \wma$ is a manifold, and the map $V \rightarrow \wma$ is a surjective submersion, hence $\wma$ is an analytic stack.
	The same argument shows that point $(3)$ follows from point $(4)$.

	To prove $(4)$ and $(5)$, take a map $B\rightarrow \Mgma$ corresponding to an $m$-pointed genus $g$ Riemann surface $(\Sigma\rightarrow B, \ua)$.
	The fibre product $B \times_{\Mgma}\wmda$ is representable by the product (over $B$), for $1 \leq i \leq m$, of the spaces $\IT^{\tc, \leq p_i, \bd_i}_{B, a_i}$ parameterising irregular types with pole at $a_i$ of order bounded by $p_i$, and with root order $\bd_i$.
	We have seen in~\cref{locstr-str} that each space $\IT^{\tc, \leq p_i, \bd_i}_{B, a_i}$ is a manifold that, locally on $B$, is of the form $B\times \prod_{j=1}^{p_i} C_j$, where $C_j$ is the manifold defined in \eqref{eq-cj}.
	The latter coincides with the one appearing in~\cite[Proposition 2.1]{drt26}, hence the proof is complete.
\end{proof}

\begin{corol}
	\label{repn-mgm}

	If $\Mgma$ is (representable by) a manifold then $\wmda$ is (representable by) a manifold as well, and the map $\wmda \rightarrow \Mgma$ is a holomorphic submersion.
\end{corol}

\begin{rem}
	For $g=0$, fibre bundles of irregular types with poles of bounded order at distinct points of $\mathbf{P}^1$, and with given pole orders after evaluation at each root, are constructed in a more direct way in \cite[Proposition 4.1, Definition 4.2]{yam19}. The stacks $\zwmda$ are isomorphic to quotients of the spaces in \cite[Definition 4.2]{yam19} by the action of the group of automorphisms of $\mathbf{P}^1$ fixing infinity (cf. \cref{g0m1} for the case $m=1$).
\end{rem}

\section{Wild mapping class groups}

In this section we define and study the fundamental groups of the stacks of admissible (families of) wild Riemann surfaces.

\subsection{Definition and basic properties}

Fix $g, m, \up$ and $\ud$ as in Theorem~\ref{mainthm-an}.
By Theorem~\ref{mainthm-an} we know that $\wmda$ is an analytic stack.
To such an object, once a base point is chosen, one can attach a (topological) fundamental group, cf.~Definition~\ref{fundgp-stack}.

\subsubsection{Independence of $\up$}
\label{indep-p}

Let us suppose in this subsection that $g \geq 1$, or that $g=0$ and $m \geq 3$.
Then $\Mgma$ is the stacky quotient of the Teichmüller space $T_{g, m}$ by the (pure) mapping class group $\Gamma_{g, m}$ (cf.~\cref{def-mgma}).
Since $T_{g, m}$ is contractible and $\Gamma_{g, m}$ is discrete we have $\pi_1(\Mgma)=\Gamma_{g, m}$, and all higher homotopy groups are trivial.

The map $\wmda\rightarrow \Mgma$ induces a map between the attached topological stacks (cf.~\cref{topst}) which, by Theorem~\ref{mainthm-an}, is a weak Serre fibration in the sense of~\cite[Definition 3.6]{noo14}.
Its fibres are products of spaces of irregular types of pole order at most $p_i$ and of root order $\bd_i$ at a point, denoted by $\IT^{\tc, \leq p_i, \bd_i}$, for $1 \leq i \leq m$. Assume that the spaces $\IT^{\tc, \leq p_i, \bd_i}$ are non-empty; then they are path connected, hence the topological stack attached to $\wmda$ is path connected (as the same is true for $\Mgma$).
Hence its fundamental group is independent of the base point, up to isomorphism; we will thus abusively omit base points from our notation.

Let us denote the fundamental group of $\prod_{i=1}^m \IT^{\tc, \leq p_i, \bd_i}$ by $\mathsf L\Gamma^{\tc, \leq \up, \ud}$, and the fundamental group of $\wmda$ by $\mathsf W\Gamma_{g, m}^{\tc, \leq \up, \ud}$.
The homotopy exact sequence~\cite[Theorem 5.2]{noo14} coming from the map $\wmda\rightarrow \Mgma$ yields a short exact sequence
\begin{equation}
	\label{homes}
	1\rightarrow \mathsf L\Gamma^{\tc, \leq \up, \ud}\rightarrow \mathsf W\Gamma_{g, m}^{\tc, \leq \up, \ud} \rightarrow \Gamma_{g, m}\rightarrow 1.
\end{equation}

For $\up' \geq \up$ (i.e., $p_i' \geq p_i$ for $1 \leq i \leq m$) we have natural maps $\IT^{\tc, \leq p_i, \bd_i}\rightarrow \IT^{\tc, \leq p_i', \bd_i}$ and $\wmda\rightarrow \mathsf{WM}_{g, m}^{\tc, \leq \up', \ud}$.
As $p_i\geq d_{\alpha, i}$ for every $i$ and every $\alpha \in \Phi$, the induced map $\mathsf L\Gamma^{\tc, \leq \up, \ud}\rightarrow \mathsf L\Gamma^{\tc, \leq \up', \ud}$ is an isomorphism (cf.~\cite[Remark 2.2]{drt26}), hence chasing the diagram consisting of the short exact sequences \eqref{homes} for $\up$ and $\up'$ we deduce that the map $\mathsf W\Gamma_{g, m}^{\tc, \leq \up, \ud}\rightarrow \mathsf W\Gamma_{g, m}^{\tc, \leq \up', \ud}$ is an isomorphism. It follows that (the isomorphism class of) $\mathsf W\Gamma_{g, m}^{\tc, \leq \up, \ud}$ does not depend on $\up$, as long as $p_i \geq \max(\bd_i)$ for $1 \leq i \leq m$.

We will see below that this is also true in the remaining cases $g=0, m\in \{1, 2\}$.
Therefore, we may give the following definition.

\begin{defin}
	Given integers $g\geq 0$, $m\geq 1$, and $\ud=(d_{\alpha, i})_{1\leq i \leq m, \alpha \in \Phi}$ with $d_{\alpha,i} \geq 0$, the (pure) wild mapping class group is
	\begin{equation*}
		\mathsf W\Gamma_{g, m}^{\tc, \ud}=\pi_1(\wmda, x),
	\end{equation*}
	for any base point $x$ and any tuple $\up=(p_1, \ldots, p_m)\in \Z^m$ such that $p_i \geq \max(\bd_i)$ for $1 \leq i \leq m$.
\end{defin}

\subsubsection{Local vs global mapping class group}
\label{lvgmcg}

In the situation of \cref{indep-p}, the short exact sequence \eqref{homes} can thus be written
\begin{equation*}
	1\rightarrow \mathsf L\Gamma^{\tc, \ud}\rightarrow \mathsf W\Gamma_{g, m}^{\tc, \ud} \rightarrow \Gamma_{g, m}\rightarrow 1.
\end{equation*}
It expresses the wild mapping class group as an extension of the usual mapping class group by a product of ``local'' wild mapping class groups, studied in \cite{drt26}.

\section{Examples}

\subsection{A chart of $\wmda$}
\label{sect-teichmueller}

As in \cref{indep-p}, let us suppose that $g \geq 1$, or $g=0$ and $m \geq 3$.
By Theorem~\ref{mainthm-an} and its proof, letting $\widetilde{WT}_{g, m}^{\tc, \leq \up, \ud}=\wmda \times_{\Mgma}T_{g, m}$, the top row of the following Cartesian diagram provides an atlas of $\wmda$:
\begin{center}
	\begin{tikzcd}
		\widetilde{WT}_{g, m}^{\tc, \leq p, \ud} \arrow[r] \arrow[d]& \wmda \arrow[d]\\
		T_{g, m} \arrow[r]
		& \Mgma.
	\end{tikzcd}
\end{center}

By Theorem~\ref{mainthm-an} we know that $\widetilde{WT}_{g, m}^{\tc, \leq p, \ud}$ is a complex manifold, and it is a locally trivial fibration over $T_{g, m}$ whose fibres are products of moduli spaces $\IT^{\tc, \leq p_i, \bd_i}$ of irregular types of pole order at most $p_i$ and of root order $\bd_i$ at a point, for $1 \leq i \leq m$.

Better, we have the following statement, which we include for completeness even though it will not be needed later.

\begin{lem}
	The fibration $\widetilde{WT}_{g, m}^{\tc, \leq \up, \ud}\rightarrow T_{g, m}$ is trivial, i.e., there is a biholomorphism
	\begin{equation*}
		T_{g, m}\times \prod_{i=1}^m \IT^{\tc, \leq p_i, \bd_i} \simeq \widetilde{WT}_{g, m}^{\tc, \leq \up, \ud}.
	\end{equation*}
\end{lem}

\begin{proof}
	It suffices to prove that $\wma \times_{\Mgma}T_{g, m}$ is trivial as a (holomorphic) vector bundle over $T_{g, m}$.
	This follows from the fact that $T_{g, m}$ is a contractible domain of holomorphy~\cite[Theorem 2]{beeh64} and from~\cite[Theorem 5.3.1]{for17}.
\end{proof}

\subsection{Genus zero examples}

Let us describe in this section the stacks $\wmda$ in the case $g=0$. Besides providing the first concrete examples of the spaces we defined, the description below will be useful in \cref{sect-wildalg} to determine which of the stacks $\wmda$ - or rather, algebraic analogues thereof -  are Deligne--Mumford.

\subsubsection{The case $m \geq 3$}

If $m \geq 3$ then $\zMgma$ is a complex manifold and by Corollary~\ref{repn-mgm} $\zwmda$ is a manifold as well.
In addition, we claim that $\zwmda\rightarrow \zMgma$ is a trivial fibration, hence we have an isomorphism
\begin{equation}\label{eq-wmcg_as_direct_product}
	\mathsf W\Gamma_{0, m}^{\tc, \ud}\simeq \mathsf L\Gamma^{\tc, \ud} \times \Gamma_{0, m}.
\end{equation}

To prove our claim, we recall first of all that every Riemann surface $\Sigma \rightarrow B$ of genus zero with $m$ disjoint sections admits a unique isomorphism to $(\mathbf{P}^1, (0, 1, \infty, \sigma_4, \ldots, \sigma_{m}))$, with $\sigma_i \colon B \rightarrow \mathbf{P}^1 \smallsetminus \{0, 1, \infty\}$ holomorphic.
Hence $\zMgma$ is isomorphic to $\mathrm{Conf}_{m-3}(\C\smallsetminus \{0, 1\})$, i.e., the complement of the diagonals in $(\C\smallsetminus \{0, 1\})^{m-3}$.
The universal family over $\zMgma$ is $\mathbf{P}^1\times \zMgma$, with sections $(\sigma_1, \ldots, \sigma_m)$ induced by maps $\varphi_i \colon \mathrm{Conf}_{m-3}(\C\smallsetminus \{0, 1\}) \rightarrow \mathbf{P}^1$ for $1 \leq i \leq m$ described as follows. Denoting the elements of $\mathbf{P}^1\times \zMgma$ by $(z, x_1, \ldots, x_{m-3})$, the map $\varphi_1$ (resp. $\varphi_2, \varphi_3$) is constantly equal to 0 (resp. 1, $\infty$) and, for $i \geq 4$, we have $\varphi_i(x_1, \ldots, x_{m-3})=x_{i-3}$.
In particular, for each $i \in \{1, \ldots, m\}$, there is a holomorphic function $f$ on a neighbourhood of $\sigma_i(\zMgma)\subset \mathbf{P}^1\times \zMgma$ whose vanishing locus is the divisor    $\sigma_i(\zMgma)$. Indeed, we can take $f=z$ (resp. $z-1$, resp. $\frac{1}{z}$) for $i=1$ (resp. $i=2, 3$), and $f=z-x_{i-3}$ for $i \geq 4$.
The discussion in~\cref{tails} (cf. in particular Lemma \ref{push-free} and Lemma \ref{sect-free}) implies that $\zwma$ is a trivial vector bundle on $\zMgma$, and $\zwmda\rightarrow \zMgma$ is a trivial fibration.

This can also be checked directly by hand as follows. Let $\Sigma=\mathbf{P}^1\times \zMgma$, let $B=\zMgma$ and let $\pi \colon \Sigma \rightarrow B$ be the projection.
It suffices to show that, for $1 \leq i \leq m$, the bundle of irregular types on $\Sigma$ with pole at $\sigma_i$ of order at most $p_i$ is trivial.
In turn, to prove this it is enough to prove that $\pi_*\T_{\Sigma, \sigma_i}^{\leq p_i}$ is isomorphic to $O_B^{p_i}$; we will do this for $i \geq 4$ (the argument for $i \leq 3$ is similar).
Fix such an $i$, and consider the holomorphic function $w=z-x_{i-3}$ on $\C \times B$, whose vanishing locus is the image of $\sigma_i$. For every open $U \subset B$, meromorphic functions on $\C \times U$ having a pole of order at most $p_i$ along $\sigma_i$ can be written, up to holomorphic terms, in the form $g=\sum_{j=1}^{p_i}\frac{1}{w^j}g_j(x_1, \ldots, x_{m-3})$, where the functions $g_j$ are holomorphic on $U$. The map sending $g$ to $(g_1, \ldots, g_{p_i})$ yields an isomorphism $\pi_*\T_{\Sigma, \sigma_i}^{\leq p_i}(U)\xrightarrow{\sim} O_B(U)^{p_i}$.

\begin{rem}
Compared to the standard ``tame'' theory of mapping class groups of pointed spheres,
the new ingredient of~\eqref{eq-wmcg_as_direct_product} is a direct product of pure local wild mapping class groups indexed by the marked points; in turn, as explained in \cite[Theorem 3]{drt26}, each pure local wild mapping class group factors as a product of pure braid groups of complex hyperplane arrangements.
\end{rem}

\subsubsection{}

For $m\in \{1, 2\}$, let $G_m$ be the complex Lie group of automorphisms of $\mathbf{P}^1$ fixing $m$ distinct chosen points.
Concretely, as we will recall later, $G_1$ is isomorphic to the affine group $\C \rtimes \C^\times$ and $G_2$ is isomorphic to $\C^\times$.
The stack $\zMgma$ is isomorphic to the quotient stack $[*/G_m]$, hence we obtain a Cartesian diagram
\begin{center}
	\begin{tikzcd}
		\prod_{i=1}^m \IT^{\tc, \leq p_i, \bd_i} \arrow[r] \arrow[d]& \zwmda \arrow[d]\\
		\{*\} \arrow[r]
		& \zMgma
	\end{tikzcd},
\end{center}
and the top horizontal map is an atlas of $\zwmda$.
We will now describe $\zwmda$ more explicitly, treating separately the cases $m=1$ and $m=2$.

\subsubsection{The case $m=1$}\label{g0m1}
Fix an integer $p \geq 1$, and fix $\mathbf{d}=(d_\alpha)_{\alpha \in \Phi}$ with $0 \leq d_\alpha \leq p$.
For every complex manifold $B$, an irregular type on $(B \times \mathbf{P}^1, \mathrm{Id} \times 0)$ of pole order at most $p$ is a section of $\tc\otimes_\C (z^{-p}\Oo_{B \times \mathbf{P}^1}/\Oo_{B \times \mathbf{P}^1})\simeq \tc\otimes_\C \Oo_B^p$.
Therefore, such irregular types are in bijection with holomorphic functions from $B$ to the space $\IT^{\tc, \leq p}$ of irregular types with a pole of order at most $p$ at a point; irregular types which in addition have root order $\mathbf{d}$ are in bijection with holomorphic functions from $B$ to $\IT^{\tc, \leq p, \mathbf{d}}$.
It follows that $\IT^{\tc, \leq p, \mathbf{d}}$ can be identified with the stack sending $B$ to the groupoid of quadruples $(\pi\colon \Sigma \rightarrow B, a, Q, \psi)$ consisting of an admissible family of wild Riemann surfaces of genus zero with one section, pole order at most $p$ and root order $\mathbf{d}$, together with an isomorphism $\psi \colon \Sigma \rightarrow B \times \mathbf{P}^1$ over $B$ sending $a$ to $Id \times 0$.
Precisely, such an identification is obtained sending $(\pi\colon \Sigma \rightarrow B, a, Q, \psi)$ to $(\psi^{-1})^*(Q)$.

The group $G_1$ therefore acts on $\IT^{\tc, \leq p, \mathbf{d}}$ modifying $\psi$, i.e., it acts by pullback on irregular types.
Quotienting out by this action amounts to forgetting the choice of isomorphism $\psi$; as such an isomorphism always exists locally on $B$, we find that
\begin{equation}
	\label{wm01}
	\mathsf{WM}_{0, 1}^{\tc, \leq p, \mathbf{d}}=[\IT^{\tc, \leq p, \mathbf{d}}/G_1].
\end{equation}
More precisely, a $B$-point of the quotient stack on the right-hand side corresponds by definition to a $G_1$-torsor $T$ over $B$ with a $G_1$-equivariant map $\varphi\colon T \rightarrow \IT^{\tc, \leq p, \mathbf{d}}$, which yields an irregular type $Q_T$ over $(T\times \mathbf{P}^1, \mathrm{Id}\times 0)$.
The $G_1$-torsor $T$ gives rise to a pointed Riemann surface of genus zero $(\Sigma \rightarrow B, a)$, obtained gluing constant families of $\mathbf{P}^1$ over open trivialising subspaces $B_i\subset B$, along isomorphisms obtained comparing trivialisations.
Similarly, the irregular type $Q_T$ restricts to a $G_1$-equivariant irregular type $Q_{T, i}$ over each $(B_i\times G_1)\times \mathbf{P}^1$.
These irregular types are uniquely determined by their restrictions $Q_i$ to $B_i \times \{1\} \times \mathbf{P}^1$, which glue to a family of irregular types $Q$ over $(\Sigma \rightarrow B, a)$.

Note that \eqref{wm01} implies that the fundamental group of $\mathsf{WM}_{0, 1}^{\tc, \leq p, \mathbf{d}}$ does not depend on the choice of $p \geq \max(\mathbf{d})$: indeed, the fibration $\IT^{\tc, \leq p, \mathbf{d}}\rightarrow \mathsf{WM}_{0, 1}^{\tc, \leq p, \mathbf{d}}$ gives rise to an exact sequence (cf. \cite[Example 5.6]{noo14})
\begin{equation*}
	\Z \simeq \pi_1(G_1)\rightarrow \pi_1(\IT^{\tc, \leq p, \mathbf{d}})\rightarrow \pi_1(\mathsf{WM}_{0, 1}^{\tc, \leq p, \mathbf{d}})\rightarrow 1.
\end{equation*}
The claimed independence follows, as in~\cref{indep-p}, from the analogous property of the fundamental group of $\IT^{\tc, \leq p, \mathbf{d}}$.

We also see that $\mathsf W\Gamma_{0, 1}^{\tc, \mathbf{d}}$ is a quotient of a pure local wild mapping class group by the image of $\pi_1(G_1)$ via the map coming from the action of $G_1$ on the base point.

Let us now describe the stack $[\IT^{\tc, \leq p, \mathbf{d}}/G_1]$; to simplify our notation, we regard $\IT^{\tc, \leq p, \mathbf{d}}$ as a moduli space of irregular types at $\infty$, rather than at $0$ as above.
Hence we can identify $G_1$ with the group of affine transformations $z \mapsto rz+s$, with $r \in \C^\times$ and $s\in \C$.
The action of $(s, r) \in \C \rtimes \C^\times \simeq G_1$ on $\IT^{\tc, \leq p, \mathbf{d}}$ sends a point corresponding to an irregular type $Q=\sum_{j=1}^pA_jz^j$ to the irregular type obtained replacing $z$ by $rz+s$.
So
\begin{equation*}
	(s, r)\cdot Q=r^pA_pz^p+\sum_{j=1}^{p-1}A'_jz^j.
\end{equation*}

Let us now distinguish two subcases:

\begin{enumerate}
	\item for $p=1$, an element $(s, r)$ acts by sending $A_1$ to $rA_1$. Indeed, $rA_1z+sA_1$ differs from $rA_1z$ by a term holomorphic at infinity;

	\item for $p\geq 2$ instead we look at the subleading coefficient:
	      \begin{equation*}
		      (s, r)\cdot Q - r^pA_pz^p = (pr^{p-1}sA_p+r^{p-1}A_{p-1})z^{p-1}+\sum_{j=1}^{p-2}A'_jz^j.
	      \end{equation*}
\end{enumerate}

In the first case $\C\subset G_1$ acts trivially on $\IT^{\tc, \leq 1}$; in particular, stabilisers in $G_1$ of points of $\IT^{\tc, \leq 1}$ are never finite.
Thus, given $\mathbf{d}=(d_\alpha)_{\alpha \in \Phi}$ with $0 \leq d_\alpha \leq 1$, for any $p\geq 1$ the space $\IT^{\tc, \leq p, \mathbf{d}}$, if non-empty, contains points with stabilisers of infinite order.

In the second case the identity $(s, r)\cdot Q = Q$ forces $r$ to be a $p$-th root of unity, if $A_p \neq 0$.
Furthermore, for each possible $r$, there is at most one $s$ such that $(pr^{p-1}sA_p+r^{p-1}A_{p-1})=A_{p-1}$.
Hence, stabilisers of points in strata $\mathsf{WM}_{0, 1}^{\tc, \leq p, \mathbf{d}}$ are finite whenever $\max(\mathbf{d})\geq 2$.

More precisely, when $\max(\mathbf{d})\geq 2$ the stack $\mathsf{WM}_{0, 1}^{\tc, \leq p, \mathbf{d}}$ is isomorphic to a locally closed substack of a weighted projective stack. To see this, choose $\alpha \in \Phi$ such that $d_{\alpha}\geq 2$, and let $Z_\alpha \subset \IT^{\tc, \leq p, \mathbf{d}}$ be the submanifold of irregular types such that $\alpha(A_{d_{\alpha}-1})=0$. Consider the map $\IT^{\tc, \leq p, \mathbf{d}}\rightarrow Z_\alpha$ sending $Q$ to $(s(Q), 1)\cdot Q$, where $s(Q)=-\frac{\alpha(A_{d_\alpha-1})}{p\alpha(A_{d_\alpha})}$. It induces a map $\IT^{\tc, \leq p, \mathbf{d}}/\C\rightarrow Z_\alpha$ inverse to the one induced from the inclusion $Z_\alpha \subset \IT^{\tc, \leq p, \mathbf{d}}$. Therefore the map $Z_\alpha \rightarrow \IT^{\tc, \leq p, \mathbf{d}}/\C$ is a biholomorphism, hence we find that $\mathsf{WM}_{0, 1}^{\tc, \leq p, \mathbf{d}}=[Z_\alpha/\C^\times]$. The group $\C^\times$ acts on $Z_\alpha$ as follows: an element $r \in \C^\times$ sends an irregular type with coefficients $(A_1, \ldots, A_p)$ to the irregular type with coefficients $(rA_1, \ldots, r^pA_p)$. Therefore $[Z_\alpha/\C^\times]$ is locally closed substack of a weighted projective stack.

\subsubsection{The case $m=2$}\label{g0m2}

Fix $\up=(p_1, p_2)\in \Z_{\geq 0}^2 \smallsetminus \{(0, 0)\}$, and tuples $\ud=(\mathbf{d}_1, \mathbf{d}_2)$ with $\mathbf{d}_i=(d_{\alpha, i})_{\alpha \in \Phi}$ such that $0 \leq d_{\alpha, i} \leq p_i$.
Let $\IT^{\tc, \leq \up,\ud}=\IT^{\tc, \leq p_1,\mathbf{d}_1}\times \IT^{\tc, \leq p_2,\mathbf{d}_2}$.
We identify this product with a moduli space of irregular types on $(\mathbf{P}^1, (0, \infty))$, and endow it with the resulting action of $G_2=\C^\times$. Arguing as in the previous section one shows that
\begin{equation*}
	\mathsf{WM}_{0, 2}^{\tc, \leq \up, \ud}=[\IT^{\tc, \leq \up,\ud}/G_2].
\end{equation*}
Therefore, we obtain an exact sequence
\begin{equation}
	\label{eq-exact_sequence_2_points}
	\Z\simeq \pi_1(G_2)\rightarrow \pi_1(\IT^{\tc, \leq \up,\ud})\rightarrow \pi_1(\mathsf{WM}_{0, 2}^{\tc, \leq \up, \ud})\rightarrow 1,
\end{equation}
which shows in particular that the fundamental group of $\mathsf{WM}_{0, 2}^{\tc, \leq \up, \ud}$ is independent of the choice of $\up$ (subject to the condition $p_i \geq \max(\mathbf{d}_i))$.
Note that one could also have obtained an exact sequence as above from the fibration $\mathsf{WM}_{0, 2}^{\tc, \leq \up, \ud}\rightarrow \mathsf{M}_{0, 2}$, using that $\pi_1(\mathsf{M}_{0, 2})$ is trivial and $\pi_2(\mathsf{M}_{0, 2})=\Z$.

Let us now assume that the integers $d_{\alpha, i}$ are not all equal to 0. Then, the action of $G_2$ on $\IT^{\tc, \leq \up,\ud}$ always has finite stabilisers.
Indeed, a point of $\IT^{\tc, \leq \up,\ud}$ is of the form $(Q_1, Q_2)\neq (0, 0)$, where $Q_1=\sum_{j=1}^{p_1}A_{j, 1}z^{-j}$ and $Q_2=\sum_{j=1}^{p_2}A_{j, 2}z^{j}$.
The action of $r \in \C^\times$ is obtained replacing $z$ by $rz$, hence sending $A_{j, 1}$ to $r^{-j}A_{j, 1}$ and $A_{j, 2}$ to $r^jA_{j, 2}$; therefore, stabilisers are finite groups of roots of unity.

If $\up$ is of the form $\up=(0, p_2)$ with $p_2>0$ then $\mathsf{WM}_{0, 2}^{\tc, \leq \up, \ud}$ is the quotient of $\IT^{\tc, \leq p_2, \mathbf{d}_2}$ by the action sending $A_{j, 2}$ to $r^jA_{j, 2}$. As $\mathbf{d}_2$ is not identically 0 (because we are assuming that the $d_{\alpha, i}$ are not all 0) one of the $A_{j, 2}$ is necessarily non-zero. Hence, the stack $\mathsf{WM}_{0, 2}^{\tc, \leq \up, \ud}$ is a locally closed substack of a weighted projective stack. For general $\up$, the stack $\mathsf{WM}_{0, 2}^{\tc, \leq \up, \ud}$ is representable over a weighted projective stack. Indeed, we may assume without loss of generality that  $\mathbf{d}_2$ is not identically 0; let $\up'=(0, p_2)$ and $\ud'=(0, \mathbf{d}_2)$. Forgetting the first irregular type we obtain a map $\mathsf{WM}_{0, 2}^{\tc, \leq \up, \ud}\rightarrow \mathsf{WM}_{0, 2}^{\tc, \leq \up', \ud'}$ which is representable: indeed, the same argument as in the proof of Theorem \ref{mainthm-an} shows that the fibre product of $\mathsf{WM}_{0, 2}^{\tc, \leq \up, \ud}$ with a manifold $B \rightarrow \mathsf{WM}_{0, 2}^{\tc, \leq \up', \ud'}$ is isomorphic to $\IT^{\tc, \leq p_1, \mathbf{d}_1}_{B, a_1}$.

\subsubsection{Description of certain mapping class groups in terms of braid groups} In the setting of \cref{g0m2}, suppose that $\underline{p} = (0,1)$.
A stratum $\IT^{\tc, \leq 1,\mathbf{d}_2}$ of irregular types at infinity can be described as follows. Let $\Phi' \subset \Phi$ be the set of roots $\alpha \in \Phi$ such that $d_{\alpha, 2} = 0$. Remove from the set $\ker(\Phi') \subset \mathfrak{t}$ the intersections $\ker(\beta) \cap \ker(\Phi')$, for those roots $\beta$ such that $d_{\beta, 2} = 1$; as explained in \cref{sect-root_strata} below, the resulting stratum is non-empty if and only if $\Phi'$ is a Levi subsystem of $\Phi$. In this case, the action of $\C^\times$ on $\IT^{\tc, \leq 1,\mathbf{d}_2}$ is given by scalar multiplication. The action on the base point $x$ yields a map $\gamma\colon \C^\times \rightarrow \IT^{\tc, \leq 1,\mathbf{d}_2}$, and by the exact sequence \eqref{eq-exact_sequence_2_points} the wild mapping class group $\mathsf W\Gamma_{0, 2}^{\tc, \ud}$ is the quotient of the fundamental group of $\IT^{\tc, \leq 1,\mathbf{d}_2}$ by the image of $\pi_1(\C^\times)$ via the morphism induced by $\gamma$.

Let us describe the wild mapping class group in the ``generic'' case, namely when $\mathbf d_2$ is constant equal to 1. The resulting stratum is the regular part $\mathfrak{t}^{\operatorname{reg}} \subset \mathfrak{t}$,
whose fundamental group is (by definition) the pure $\mathfrak{g}$-braid group.

Let us first assume that $\mathfrak{g}$ is simple, and fix $A_{1,2} \in \mathfrak{t}^{\operatorname{reg}}$. The pure $\mathfrak{g}$-braid group has infinite cyclic centre,
generated by the loop
\begin{equation}\label{eq-loop}
	t \mapsto e^{2\pi \sqrt{-1} \cdot t} A_{1,2},
	\quad t \in [0,1].
\end{equation}
Indeed, this holds more generally for the pure braid groups of finite irreducible complex reflection groups~\cite[Theorem~1.2]{dmm11}. Therefore,
the wild mapping class group $\mathsf W\Gamma_{0, 2}^{\tc, \ud}$ is the quotient of the pure $\mathfrak{g}$-braid group by its centre.

If $\mathfrak{g} \simeq \prod_{i \in I} \mathfrak{I_i}$ is semisimple, then the wild mapping class group is still the quotient of $\pi_1 (\mathfrak{t}^{\operatorname{reg}})$ by the subgroup generated by the loop \eqref{eq-loop}; this loop is a product of generators of the centres of the pure braid groups attached to each factor.

More generally, suppose that $\mathfrak{g}$ is reductive with centre $\mathfrak Z$ and semisimple part $\mathfrak{g'} = [\mathfrak{g},\mathfrak{g}] \subset \mathfrak{g}$. One has $\mathfrak{t}^{\operatorname{reg}} = \mathfrak Z \times (\mathfrak{t}')^{\operatorname{reg}}$ for a Cartan subalgebra $\mathfrak{t}' \subset \mathfrak{g'}$. The space $\mathfrak{t}^{\operatorname{reg}}$ deformation-retracts onto $(\mathfrak{t}')^{\operatorname{reg}} \subset \mathfrak{t}'$, and the description of the wild mapping class group in this case reduces to the semisimple case.

\subsubsection{Some isomorphisms between moduli spaces of wild Riemann surfaces of genus 0}

Fix integers $n \geq 0$ and $m \geq 0$, and let $\mathfrak{t}_{n+1} \subset \mathfrak{sl}_{n+1}$ be the subalgebra of diagonal matrices of trace zero. Let $\mathfrak{t}_{n+1}^{\mathrm{reg}} \subset \mathfrak{t}_{n+1}$ be the subspace of matrices with distinct entries. We will construct an isomorphism
\begin{equation}\label{eq-exciso}
	\mathsf{WM}_{0, m+2}^{\tc_{n+1}, \leq \underline{p}, \ud} \xrightarrow{\sim} \mathsf{WM}_{0, n+2}^{\tc_{m+1}, \leq \underline{p'}, \underline{\mathbf{d}'}},
\end{equation}
where $\underline{p}$ (resp. $\underline{p'}$) consists of one 1 followed by $m+1$ (resp. $n+1$) 0, and $\mathbf{d}_{1}$, $\mathbf{d}_{1}'$ are identically equal to 1. This rests on the following elementary lemma, where $X_n$ denotes the complement of the diagonals in $(\C^\times)^n$.
\begin{lem}
	For $n \geq 1$, if $\mathrm{diag}(b_1, \ldots, b_{n+1})$ belongs to $\tc_{n+1}^{\mathrm{reg}}$ then $(b_2-b_1, \ldots, b_{n+1}-b_1)$ belongs to $X_n$. The resulting map $\varphi_n: \tc_{n+1}^{\mathrm{reg}} \rightarrow X_n$ is a biholomorphism.
\end{lem}
\begin{proof}
	The map sending $(x_1, \ldots, x_n)$ to $(b_1, \ldots, b_{n+1})$  where $b_1=-\frac{1}{n+1}(x_1+\ldots +x_n)$ and $b_i=b_1+x_{i-1}$ for $i \geq 2$ is the inverse of $\varphi_n$.
\end{proof}

We obtain a biholomorphism
\begin{equation}\label{eq-framediso}
	(\varphi_n, \varphi_m^{-1}): \tc_{n+1}^{\mathrm{reg}}\times X_m \rightarrow X_n \times \tc_{m+1}^{\mathrm{reg}}.
\end{equation}
The left-hand side is the moduli space of $(m+2)$-tuples of distinct points on $\mathbf{P}^1$ of the form $(a_1=0, a_2, \ldots, a_{m+1}, a_{m+2}=\infty)$, together with a $\tc_{n+1}$-valued irregular type at 0 of pole order one and root order one at each root. The right-hand side has a similar modular interpretation, with $m$ and $n$ switched. The group $\C^\times$ of automorphisms of $\mathbf{P}^1$ fixing $0$ and $\infty$ acts on the two spaces in \eqref{eq-framediso}: the action of $\alpha \in \C^\times$ on $X_m$ pulling back points sends $(x_1, \ldots, x_m) \in X_m$ to $(\alpha^{-1}x_1, \ldots, \alpha^{-1}x_m)$, and the action pulling back irregular types sends $\frac{Q}{z}$ to $\frac{Q}{\alpha z}$, hence it is given by multiplication by $\alpha^{-1}$ on $ \tc_{n+1}^{\mathrm{reg}}$. For the same reason $\alpha$ acts by multiplication by $\alpha^{-1}$ on the right-hand side of \eqref{eq-framediso}, hence the map in \eqref{eq-framediso} induces the isomorphism in \eqref{eq-exciso}.

\subsection{Genus one examples}

We have $\mathsf{M}_{1, 1}=[\mathbf{H}/\mathrm{SL}_2(\Z)]$, where $\mathbf{H}=\{\tau \in \C \mid \mathrm{Im}(\tau)>0\}$ and $\mathrm{SL}_2(\Z)$ acts on it via Möbius transformations.
\subsubsection{} Over $\mathbf{H}$ we have an elliptic curve (i.e., a family of compact 1-pointed Riemann surfaces of genus one)
\begin{equation*}
	\mathcal{E}=(\mathbf{H}\times \mathbf{C})/\Z^2, \; \sigma \colon \mathbf{H} \rightarrow \mathcal{E},
\end{equation*}
where $(r, s)\in \Z^2$ acts sending $(\tau, z)$ to $(\tau, z+r\tau+s)$, and $\sigma(\tau)=(\tau, 0)$.
Therefore, the fibre $\mathcal{E}_\tau$ of $\mathcal{E}$ over $\tau \in \mathbf{H}$ is the torus $\C/(\Z+\tau \Z)$ with the marked point $0$.
The action of $\mathrm{SL}_2(\Z)$ on $\mathbf{H}$ lifts to an action on the surface $\mathcal{E}$, induced by
\begin{align}
	\label{sl2Zact}
	\mathrm{SL}_2(\Z)\times (\mathbf{H}\times \mathbf{C}) & \rightarrow \mathbf{H}\times \mathbf{C}                          \\
	\nonumber \left(
	\begin{pmatrix}
		a & b \\
		c & d
	\end{pmatrix},
	(\tau, z)\right)                                      & \mapsto \left(\frac{a\tau+b}{c\tau+d}, \frac{z}{c\tau+d}\right).
\end{align}

The coordinate $z$ is a local coordinate around each $\sigma(\tau)$, hence irregular types with a pole at $\sigma$ of order at most $p$ can be written as
\begin{equation}
	\label{irrt-M11}
	\sum_{j=1}^p z^{-j}A_j(\tau),
\end{equation}
where each $A_j\colon \mathbf{H}\rightarrow \tc$ is a holomorphic function.
Therefore, we see that $\widetilde{WT}_{1, 1}^{\tc, \leq p}=\mathsf{WM}_{1,1}^{\tc, \leq p}\times_{\mathsf{M}_{1,1}}\mathbf{H}$ is isomorphic to $\mathbf{H}\times \tc^p$: an irregular type as in \eqref{irrt-M11} corresponds to the point $(\tau, A_1(\tau), \ldots, A_p(\tau))$.
Furthermore, the vector bundle $\mathsf{WM}_{1, 1}^{\tc, \leq p}$ on $\mathsf{M}_{1, 1}=[\mathbf{H}/\mathrm{SL}_2(\Z)]$ is the quotient of $\widetilde{WT}_{1, 1}^{\tc, \leq p}$ by an action of $\mathrm{SL}_2(\Z)$ which can be written down explicitly.
For instance, let us explain the case $p=1$. Take $\begin{pmatrix}
		a & b \\
		c & d
	\end{pmatrix} \in \mathrm{SL}_2(\Z)$ and let $\tau'=\frac{a\tau+b}{c\tau+d}$. The map sending $z$ to $\frac{1}{c\tau+d}z$ induces an isomorphism from $\mathcal{E}_\tau$ to $\mathcal{E}_{\tau'}$. Via this isomorphism, an irregular type $\frac{A}{(c\tau+d)z}$ on $\mathcal{E}_{\tau'}$ pulls back to the irregular type $\frac{A}{z}$ on $\mathcal{E}_\tau$.

Therefore, via the identification between $\widetilde{WT}_{1, 1}^{\tc, \leq p}$ and $\mathbf{H} \times \tc$ described above, the action of $\begin{pmatrix}
		a & b \\
		c & d
	\end{pmatrix} \in \mathrm{SL}_2(\Z)$ sends $(\tau, A)$ to $\left(\frac{a\tau+b}{c\tau+d}, \frac{1}{c\tau+d}A\right)$.
As a consequence, the stack $\mathsf{WM}_{1, 1}^{\tc, \leq 1}$ is the quotient of $\widetilde{WT}_{1, 1}^{\tc, \leq 1}\simeq \mathbf{H}\times \tc$ by the $\mathrm{SL}_2(\Z)$-action given by
\begin{equation*}
	\left(\begin{pmatrix}
		a & b \\
		c & d
	\end{pmatrix},
	(\tau, A)\right) \mapsto \left(\frac{a\tau+b}{c\tau+d}, \frac{1}{c\tau+d}A\right).
\end{equation*}
In other words, $\mathsf{WM}_{1, 1}^{\tc, \leq 1}$ is isomorphic to the product (over $\mathsf{M}_{1, 1}$) of $r$ copies of the line bundle $\mathcal{L}^{-1}=(\mathbf{H}\times \C)/\mathrm{SL}_2(\Z)$, where the action of $\mathrm{SL}_2(\Z)$ is given by \eqref{sl2Zact}.
Its inverse $\mathcal{L}$ is isomorphic to the Hodge bundle on $\mathsf{M}_{1, 1}$ (cf.~\cite[\S~5.4]{ha11}), which generates the Picard group of $\mathsf{M}_{1, 1}$, isomorphic to $\Z/12\Z$~\cite[Theorem 6.6]{ha11} (see also~\cite{mum65}).

In particular, $\mathsf{WM}_{1, 1}^{\tc, \leq 1}\rightarrow \mathsf{M}_{1, 1}$ is a non-trivial vector bundle.

\subsubsection{The case $p=1, \mathfrak{g}=\mathfrak{sl}_2$}

Let $p=1$ as above.
Let us take $\tc$ to be the diagonal matrices in $\mathfrak{g}=\mathfrak{sl}_2$, and let $\alpha$ be the usual choice of a positive root.
Let $B \rightarrow \mathsf{M}_{1, 1}$ be a map corresponding to an elliptic curve $(\pi \colon \Sigma \rightarrow B, \sigma)$.
By Theorem~\ref{mainthm-an} the fibre product $\mathsf{WM}_{1, 1}^{\tc, \leq 1}\times_{\mathsf{M}_{1, 1}}B$ is isomorphic to the vector bundle attached to the locally free sheaf $\pi_*(\mathcal{I}_\sigma^{-1}/\mathcal{O}_{\Sigma})\simeq \sigma^*(\mathcal{I}_\sigma^{-1}/\mathcal{O}_{\Sigma})$, which is the inverse of $\sigma^*(\mathcal{I}_\sigma/\mathcal{I}_\sigma^2)\simeq \sigma^*(\Omega^1_{\Sigma/B})$ (where the isomorphism follows from the conormal exact sequence).
We see in particular once again that $\mathsf{WM}_{1, 1}^{\tc, \leq 1}$ corresponds to the inverse of the Hodge bundle over $\mathsf{M}_{1, 1}$.

Let us now consider the substack $\mathsf{WM}_{1, 1}^{\tc, \leq 1, 1}\subset \mathsf{WM}_{1, 1}^{\tc, \leq 1}$ parametrising admissible families of wild Riemann surfaces with pole order 1 at $\alpha$.
In other words, this is the complement of the zero section inside $\mathsf{WM}_{1, 1}^{\tc, \leq 1}$.
Explicitly, we have $\mathsf{WM}_{1, 1}^{\tc, \leq 1, 1}\simeq (\mathbf{H}\times \C^\times)/\mathrm{SL}_2(\Z)$, where the action is given by the formula in \eqref{sl2Zact}.
Note that this action is free, and the stack $\mathsf{WM}_{1, 1}^{\tc, \leq 1, 1}$ is a complex manifold.
The short exact sequence in~\cref{lvgmcg} reads
\begin{equation*}
	1\rightarrow \Z\rightarrow \mathsf W\Gamma_{1 ,1}^{\tc, 1} \rightarrow \mathrm{SL}_2(\Z)\rightarrow 1.
\end{equation*}

Finally, the map $\mathbf{H}\times \C^\times \rightarrow \mathbf{H}\times \C^\times$ sending $(\tau, z)$ to $(\tau, z^{-1})$ induces an isomorphism between $\mathsf{WM}_{1, 1}^{\tc, \leq 1, 1}$ and the complement $\mathcal{L}^\times$ of the zero section inside the Hodge bundle. The fundamental group of $\mathcal{L}^\times$ is isomorphic to the braid group $B_3$~\cite[Corollary 8.3]{ha11}, therefore we find that
\begin{equation*}
	\mathsf W\Gamma_{1 ,1}^{\tc, 1}\simeq B_3.
\end{equation*}

\section{Stacks of wild algebraic curves}\label{sect-wildalg}

In this section we explain how to adapt the above discussion to define algebraic stacks $\wmd$ of (admissible families of) wild algebraic curves; this is straightforward, hence we will be brief. Our main goal is to determine which of the stacks $\wmd$ are Deligne--Mumford, confirming the expectation of \cite{bdr22} in the untwisted case. We will work in the category $\sch$ of schemes over $\Spec(\C)$, and $S$ will always denote such a scheme. If $\mathfrak{g}, \mathfrak{t}$ and the roots $\alpha$ are defined over a field $F$, the construction of $\wmd$ works over $\Spec(F)$.

\subsection{Wild algebraic curves}\label{subsec-wac} By a curve over $S$ of genus $g$ we mean a smooth, proper morphism of schemes $\pi \colon C \rightarrow S$ of relative dimension one, such that all the fibres are geometrically connected of genus $g$. Note that $C/S$ is a smooth curve in the sense of \cite[\S(1.2.1)]{kama85}. By \cite[Lemma 1.1.2]{kama85}, if $\sigma \colon S \rightarrow C$ is a section then the image $\sigma(S)$ is an effective Cartier divisor, i.e., its ideal sheaf $\mathcal{I}_\sigma$ is invertible, generated locally by a nonzerodivisor. Therefore, we can define $\T_{C, \sigma}^{\leq k}$, for $k \geq 0$, as in \cref{tails}. The results in the rest of \cref{tails}, which rely on the purely algebraic Lemma \ref{sect-free}, carry over to this setting.

We define an irregular type (on $S$) with pole at $\sigma$ of order at most $k$ to be a global section of the sheaf $\tc\otimes_\C \T_{C, \sigma}^{\leq k}$. The functor  $\IT^{\tc, \leq k}_{S, \sigma}$ from $S$-schemes to sets sending $S'\rightarrow S$ to $\Gamma(C', \tc\otimes \T_{C', \sigma'}^{\leq k})$, where $C'=C\times_S S'$ and $\sigma'$ is the pullback of $\sigma$, is representable by a vector bundle over $S$, denoted by the same symbol. This is shown as Proposition \ref{irrt-vectbun} using \cite[\href{https://stacks.math.columbia.edu/tag/02KG}{Tag 02KG}]{stacks-project} to prove the analogue of Lemma \ref{bc-onB}, and \cite[Proposition 11.3, \S 11.4]{gowe20}.

\begin{defin}
	A wild algebraic curve of genus $g$ over a scheme $S$ with $m$ marked points, and pole orders bounded by $\up=(p_1, \ldots, p_m)$, is a triple $(\pi \colon C \rightarrow S, \ua, \uQ)$ consisting of a curve $\pi \colon C \rightarrow S$ of genus $g$, an $m$-tuple $\ua=(a_1, \ldots, a_m)$ of mutually non-intersecting sections $a_i \colon S \rightarrow C$ of $\pi$, and an $m$-tuple $\uQ=(Q_1, \ldots, Q_m)$ of irregular types $Q_i \in \Gamma(C, \tc\otimes_\C \T_{C, a_i}^{\leq p_i})$.

	We denote by $\wm(S)$ the groupoid of wild algebraic curves over $S$ of genus $g$ and pole orders bounded by $\up$ (morphisms being isomorphisms of $S$-schemes respecting the sections and the irregular types).
\end{defin}

We obtain a 2-functor $\wm \colon \sch \rightarrow \mathrm{Groupoids}$ sending $S$ to $\wm(S)$, and a map $\wm \rightarrow \Mgm$ forgetting irregular types, where the target is the stack of $m$-pointed algebraic curves of genus $g$. The argument in the proof of Theorem \ref{mainthm-an}(2) shows that this map is representable: if $S \rightarrow \Mgm$ corresponds to a pointed curve $(C, \ua)$ then $\wm\times_{\Mgm} S=\prod_{i=1}^m \IT^{\tc, \leq p_i}_{S, a_i}$, where the product is fibred over $S$.

\begin{ex}
	Let $\mathfrak{g}=\mathfrak{sl}_2$ and $\up=(1, 1, \ldots, 1)$. For every $(\pi: C \rightarrow S, \ua)$, we have $\T_{C, a_i}^{\leq 1}=\mathcal{I}_{a_i}^{-1}/\Oo_{C}$. Seen as a sheaf on $S$, this is isomorphic to the inverse of $a_i^*(\Omega^1_{C/S})$ by  \cite[\href{https://stacks.math.columbia.edu/tag/06BB}{Tag 06BB}]{stacks-project}. Therefore, in this case the line bundles of irregular types at each marked point generate a free subgroup of $\mathrm{Pic}(\Mgm)$ for $g \geq 3$; cf. \cite[Theorem 2]{ac87}, where the line bundles $a_i^*(\Omega^1_{C/S})$ are denoted by $\psi_i$.
\end{ex}

\begin{lem}
	$\wm$ is an algebraic stack for the $fppf$ topology.
\end{lem}
\begin{proof}
	The properties in the statement hold true for $\Mgm$, thanks to \cite[Proposition 4.5]{lmb00} and \cite[Example 4.4.13, Theorem 8.4.5]{ols16} (resp. \cite[Theorem 13.1.2]{ols16}) for $g\neq 1$ (resp. $g=1$; recall that $m\geq 1$ in this text). The fact that $\wm$ is a stack follows as in the proof of Proposition \ref{wmstack}, using \cite[\href{https://stacks.math.columbia.edu/tag/023T}{Tag 023T}]{stacks-project}. Finally, as the map $\wm \rightarrow \Mgm$ is representable, by \cite[Proposition 4.5]{lmb00} the stack $\wm$ is algebraic.
\end{proof}

\subsubsection{Admissible wild algebraic curves} We stratify $\wm$ according to the order of pole of the irregular types after evaluation at roots $\alpha \in \Phi$. Let $\ud=(d_{\alpha, i})_{\alpha \in \Phi, 1\leq i \leq m}$ with $0 \leq d_{\alpha,i}\leq p_i$. We define $\mathcal{WM}_{g, m}^{\tc, \leq \underline{p}, \leq \ud}$ to be the substack of $\mathcal{WM}_{g, m}^{\tc, \leq \underline{p}, \ud}$ whose $S$-points are wild algebraic curves $(\pi \colon C \rightarrow S, \ua, \uQ)$ such that, for $1 \leq i \leq m$, we have $\alpha \circ Q_i\in \Gamma(C, \T_{C, a_i}^{\leq d_{\alpha, i}})$. The map $\mathcal{WM}_{g, m}^{\tc, \leq \underline{p}, \leq \ud} \rightarrow \Mgm$ is representable: the fibre product with an $S$-point of $\Mgm$ corresponding to $(C, \ua)$ is the product over $S$ of the intersections of the schemes $\IT^{\tc, \leq p_i}_{S, a_i} \times_{\IT^{\leq p_i}_{S, a_i}}\IT^{\leq d_{\alpha, i}}_{S, a_i}$, where $\IT^{\leq d_{\alpha, i}}_{S, a_i}$ denotes the scheme obtained setting $\mathbf{t}=\mathbf{C}$ at the beginning of \cref{subsec-wac}. Each $\mathcal{WM}_{g, m}^{\tc, \leq \underline{p}, \leq \ud}$ is a closed substack of $\wm$; we define $\wmd \subset \mathcal{WM}_{g, m}^{\tc, \leq \underline{p}, \leq \ud}$ to be the complement of the union of the substacks $\mathcal{WM}_{g, m}^{\tc, \leq \underline{p}, \leq \underline{\mathbf{d}'}}$ for $\underline{\mathbf{d}'}<\ud$. Its $S$-points are wild algebraic curves $(\pi \colon C \rightarrow S, \ua, \uQ)$ such that $Q_i$ has pole order at most $p_i$, each $\alpha(Q_i)$ belongs to $\Gamma(C, \T_{C, a_i}^{\leq d_{\alpha, i}})$, and it does not belong to $\Gamma(C, \T_{C, a_i}^{\leq d_{\alpha, i}-1})$ anywhere on $S$ (i.e., the pullback of $\alpha(Q_i)$ to every fibre $C_s, s \in S$ does not belong to $\Gamma(C_s, \T_{C_s, a_{i, s}}^{\leq d_{\alpha, i}-1})$).

The fibre of the map $\wmd \rightarrow \Mgm$ over a $\C$-point is an $m$-fold product of products of hyperplane complements as in Theorem \ref{mainthm-an}(5). If this space is non-empty - a condition which only depends on $\ud$, and not on $\up$ - we say that $\ud$ is relevant.

\begin{prop}
	Fix integers $g\geq 0, m\geq 1$, and tuples
	\begin{equation*}
		\up=(p_1, \ldots, p_m) \in \Z_{\geq 0}^m, \quad \ud=(d_{\alpha, i})_{\alpha \in \Phi, 1\leq i \leq m}, \qquad 0 \leq d_{\alpha,i}\leq p_i,                                                                                                                                                  \end{equation*}
	such that $\ud$ is relevant. The stack $\wmd$ is Deligne--Mumford if and only if
	\begin{equation*}
		2g-2+m+\sum_{i}\max(\mathbf{d}_i) > 0.
	\end{equation*}
\end{prop}
\begin{proof}
	If $g \geq 2$ (resp. $g=1$) then the statement follows from \cite[Proposition 4.5]{lmb00} and the fact that $\mathcal{M}_g$ (resp. $\mathcal{M}_{1, 1}$) is Deligne--Mumford by \cite[Theorem 8.4.5]{ols16} (resp. \cite[Theorem 13.1.2]{ols16}). If $g=0$ and $m \geq 3$ then $\mathcal{M}_{0, m}$ and $\wmd$ are algebraic varieties. Suppose that $g=0$ and $m=1$. As in \cref{g0m1} we see that $\wmd$ is a quotient of a space of irregular types by the affine group. If $\max(\mathbf{d}_1)\leq 1$ then $\wmd$ has complex points with automorphism group $\mathbf{G}_a$, therefore it is not Deligne--Mumford. If $\max(\mathbf{d}_1)>1$ then the argument in \cref{g0m1} shows that $\wmd$ is a locally closed substack of a weighted projective space, hence it is Deligne--Mumford. If $m=2$ and $\ud$ is identically zero then $\wm$ has a substack isomorphic to $\mathcal{M}_{0, 2}$ (parametrizing curves with trivial irregular types). Finally, suppose that $\ud=(\mathbf{d}_1, \mathbf{d}_2)$ is not identically 0. Assume, without loss of generality, that $\mathbf{d}_2\neq 0$; as in \cref{g0m2} we have a representable map $\mathcal{WM}_{0, 2}^{\tc, \leq \up, \ud}\rightarrow \mathcal{WM}_{0, 2}^{\tc, \leq \up', \ud'}$, hence it suffices to prove that $\mathcal{WM}_{0, 2}^{\tc, \leq \up', \ud'}$ is Deligne--Mumford. This is true because $\mathcal{WM}_{0, 2}^{\tc, \leq \up', \ud'}$ can be written, for the same reason as in \cref{g0m2}, as a quotient of a space of (non-zero) irregular types by a $\mathbf{G}_m$-action, making it a locally closed substack of a weighted projective stack.
\end{proof}

\subsubsection{A parametrization of non-empty strata}\label{sect-root_strata}
One can parametrise the non-empty spaces $\wmd$ in terms of suitable filtrations of $\Phi$. To do so, we may take a fixed Riemann surface of genus $g$ with one marked point $a$, and look at a space $\IT^{\tc, \leq p}$ of irregular types with pole at $a$ of order bounded above by a given integer $p\geq 0$; we want to parametrise non-empty strata of the root stratification on $\IT^{\tc, \leq p}$.

Define a Levi subsystem of $\Phi$ to be a subset $\Phi'\subset \Phi$ such that $\mathrm{span}_\C(\Phi')\cap \Phi=\Phi'$. A sequence $(\Phi_1, \Phi_2, \ldots)$ of Levi subsystems of $\Phi$ such that $\Phi_i \subset \Phi_{i+1}$ for $i \geq 1$ will be called a Levi filtration of $\Phi$. Such a filtration is of depth $d$ if $\Phi_d \subsetneq \Phi$ and $\Phi_{d+1}=\Phi$. We denote by $\mathcal{L}_{\Phi}^{\leq p}$ the set of Levi filtrations of $\Phi$ of depth at most $p$.

Let us identify $\IT^{\tc, \leq p}$ with $\tc^p$, sending an irregular type $Q=\frac{A_1}{z}+\ldots +\frac{A_p}{z^p}$ to $(A_1, \ldots, A_p)$. We attach to $Q$ a Levi filtration obtained setting $\Phi_i=\{\alpha \in \Phi \mid \alpha(A_j)=0 \text{ if } i \leq j \leq p \}$ for $i \leq p$, and $\Phi_i=\Phi$ for $i>p$; in other words, the pole order of $\alpha(Q)$ equals $i\geq 0$ if and only if $\alpha$ belongs to $\Phi_{i+1}\smallsetminus \Phi_i$ (where we set $\Phi_0=\varnothing$). In particular, we see that the filtration $(\Phi_i)_{i \geq 1}$ only depends on the root stratum to which $Q$ belongs.

Conversely, given a Levi filtration of depth at most $p$ we define $\mathbf{d}=(d_\alpha)_{\alpha \in \Phi}$ setting $d_\alpha=i$ if $\alpha \in \Phi_{i+1}\smallsetminus \Phi_i$. To show that the stratum $\IT^{\tc, \leq p, \mathbf{d}}$ is non-empty note that, if $\Phi_i\subsetneq \Phi_{i+1}$, then for each $\alpha \in \Phi_{i+1} \smallsetminus \Phi_i$ the intersection of the kernels of the roots in $\Phi_i$ meets the complement of $\ker(\alpha)$; therefore
\begin{equation*}
	\left(\bigcap_{\alpha \in \Phi_i} \ker(\alpha)\right)\cap \left(\bigcap_{\alpha\in \Phi_{i+1}\smallsetminus \Phi_i}(\tc \smallsetminus \ker(\alpha))\right)
\end{equation*}
is non-empty (cf. \cite[Lemma 2.2.1]{cfrw}). Hence, there are $A_1, A_2, \ldots, A_p$ such that $\frac{A_1}{z}+\ldots +\frac{A_p}{z^p}$ belongs to $\IT^{\tc, \leq p, \mathbf{d}}$.

To sum up, non-empty strata of the root stratification of $\IT^{\tc, \leq p}$ are in bijection with $\mathcal{L}_{\Phi}^{\leq p}$.

\appendix

\section{Background on analytic stacks}
\label{sec:appendix}

In this appendix we recall and give references for some generalities on stacks and their fundamental group which are needed in the text.
The material in this section is standard, and is included for the reader's convenience; for the same reason we try to provide some motivation for the concept of (analytic) stack. Needless to say, this is not meant to be an introduction to the subject, for which we refer the reader to the references given below.

\subsection{Stacks}

We collect here the background material on stacks that is used in the text.
We will focus on analytic stacks, and our main references will be~\cite{noo05, beno06, hei05}.

\subsubsection{Motivation: the moduli space of sections of a locally free sheaf}
\label{funp-vectbun}

Let us start by showing the classical theory of moduli spaces in action from a modern viewpoint, in a simple example of interest in this document.

Fix a complex manifold $X$ and a locally free sheaf $\mathcal{E}$ of $\Oo_X$-modules of finite rank $n$.
For every complex manifold $B$ with a map $\varphi \colon B\rightarrow X$ we can consider the global sections $\Gamma(B, \varphi^* \mathcal{E})$ of the pullback of $\mathcal{E}$ to $B$.
We would like to construct the moduli space of all such global sections, i.e., a complex manifold $V(\mathcal{E})\xrightarrow{\pi} X$ such that there are natural bijections, for every $\varphi \colon B\rightarrow X$,
\begin{equation*}
	\Gamma(B, \varphi^* \mathcal{E})\xrightarrow{\sim} \{\sigma \colon B \rightarrow V(\mathcal{E}) \mid \pi \circ \sigma=\varphi\}.
\end{equation*}

We claim that the vector bundle $\pi \colon V(\mathcal{E})\rightarrow X$ attached to the locally free sheaf $\mathcal{E}$ satisfies this property.
Indeed, recall that holomorphic sections of $\pi$ correspond to elements of $\Gamma(X, \mathcal{E})$.
Now for every $\varphi \colon B \rightarrow X$ we can consider the pullback $\varphi^*V(\mathcal{E})\rightarrow B$; it is the subset of $B\times V(\mathcal{E})$ consisting of points $(b, v)$ such that $\varphi(b)=\pi(v)$, and the inclusion $\varphi^*V(\mathcal{E}) \subset B \times V(\mathcal{E})$ is a holomorphic immersion.
It follows that maps $\sigma \colon B \rightarrow V(\mathcal{E})$ such that $\pi \circ \sigma=\varphi$ correspond bijectively to holomorphic sections $B \rightarrow \varphi^*V(\mathcal{E})$ of the projection $\varphi^*V(\mathcal{E})\rightarrow B$.
But $\varphi^*V(\mathcal{E})$ is the vector bundle attached to $\varphi^*(\mathcal{E})$, hence such sections correspond to elements of $\Gamma(B, \varphi^*(\mathcal{E}))$.

\subsubsection{Complex manifolds as sheaves}

In~\cref{funp-vectbun}, we have described the vector bundle $V(\mathcal{E})$ attached to a locally free sheaf $\mathcal{E}$ in terms of the holomorphic maps from an arbitrary manifold $B$ to it: holomorphic maps $B\rightarrow V(\mathcal{E})$ are in natural bijection with couples $(\varphi, s)$ consisting of a holomorphic map $\varphi \colon B \rightarrow X$ and an element $s \in \Gamma(B, \varphi^*\mathcal{E})$.

In general, for every complex manifold $X$, we can look at the functor $h(X) \colon \man^{op} \rightarrow \mathrm{Sets}$ sending $B$ to the set $\mathrm{Hom}(B, X)$ of holomorphic maps from $B$ to $X$.
A holomorphic map $X \rightarrow Y$ induces a natural transformation of functors $h(X)\rightarrow h(Y)$.
This gives rise to the Yoneda embedding
\begin{align*}
	\man & \rightarrow \mathrm{Functors}(\man^{\mathrm{op}}, \mathrm{Sets}) \\
	X    & \mapsto h(X).
\end{align*}
By Yoneda's lemma, the functor $X \mapsto h(X)$ is fully faithful; in other words, we can see the category of complex manifolds as a subcategory of the category $\mathrm{Functors}(\man^{op}, \mathrm{Sets})$.
Constructing a (fine) moduli space amounts to formulating a moduli problem, i.e., writing down a functor $M \colon \man^{\mathrm{op}}\rightarrow \mathrm{Sets}$, and proving that $M$ is isomorphic to $h(X)$ for some complex manifold $X$.
If such an $X$ exists, it is unique up to isomorphism, and we say that the moduli problem $M$ is representable by $X$.

One feature distinguishing functors of the form $h(X)$ from arbitrary functors is that the former are sheaves on $\man$.
This means that, if $(B_i)_{i \in I}$ is an open cover of a manifold $B$, and $b_i \in h(X)(B_i)$ are elements agreeing on the intersections $B_i \cap B_j$, then there is a unique $b\in h(X)(B)$ restricting to $b_i$ on $B_i$.
This is because the elements $b_i$ are holomorphic maps $B_i\rightarrow X$, which can be glued to a holomorphic map $B \rightarrow X$ if they agree on all the intersections $B_i \cap B_j$.

\subsubsection{A moduli problem which is not a sheaf}

Many moduli problems of interest turn out not to be representable, often because they are not sheaves on $\man$.

For example, consider the functor $M^{/iso}_{1, 1}$ sending $B$ to the set of isomorphism classes of elliptic curves over $B$.
These are couples ($\pi \colon \Sigma\rightarrow B, \sigma$) consisting of a proper submersion $\pi$ of complex manifolds whose fibres are one dimensional complex tori, and a holomorphic section $\sigma \colon B \rightarrow \Sigma$ of $\pi$; isomorphisms are biholomorphisms commuting with the structure map to $B$ and with the sections.

We claim that $M^{/iso}_{1, 1}$ is not a sheaf on $\man$.
Indeed, take $B=\C\smallsetminus \{0\}$ and consider the product $B \times \mathbf{P}^2(\C)$, whose elements we denote by $(b, [x: y: z])$.
Let $\Sigma\subset B \times \mathbf{P}^2(\C)$ be the submanifold cut out by the equation $by^2z=x^3-z^2x$.
Let $\pi \colon \Sigma \rightarrow B$ be the map induced by projection of $B \times \mathbf{P}^2(\C)$ on the first component, and $\sigma$ be the map sending $b$ to $(b, [0: 1: 0])$.
We can write $B=B_1 \cup B_2$, where $B_1$ (resp. $B_2$) is the complement in $\C$ of the half-line $\R_{\geq 0}$ (resp. $\R_{\leq 0}$).
Over each $B_i$ we can choose a holomorphic square root function $r_i \colon B_i \rightarrow \C$.
Then, the map sending $(b, [x: y: z])$ to $(b, [x: r_i(b)y: z])$ induces an isomorphism between $\Sigma$ and the elliptic curve $\Sigma_0$ with equation $y^2z=x^3-z^2x$ over $B_i$.
However, there is no isomorphism between $\Sigma$ and $\Sigma_0$ over the whole $B$---the monodromy representation attached to $\Sigma$ is non-trivial, cf.~\cite[Proposition 3.3]{lvo19}---hence $M^{/iso}_{1, 1}$ is not a sheaf.

\subsubsection{Stacks}
\label{stacks}

The problem in the above example comes from the following phenomenon: over each $B_i$ we have an isomorphism $f_i \colon \Sigma \rightarrow \Sigma_0$, but over $B_1 \cap B_2$ the isomorphisms $f_1, f_2$ differ by a non-trivial automorphism of $\Sigma_0$ (sending $y$ to $-y$).
Hence, we see that the existence of non-trivial automorphisms of elliptic curves is intimately related to the fact that $M^{/iso}_{1, 1} \colon \man^{op}\rightarrow \mathrm{Sets}$ is not a sheaf.

Grothendieck's idea to handle this issue is to make automorphisms part of the picture: rather than looking at functors $M \colon \man^{op} \rightarrow \mathrm{Sets}$, one considers (2-)functors
\begin{equation*}
	\M \colon \man^{op} \rightarrow \mathrm{Groupoids}.
\end{equation*}
Explicitly, for every complex manifold $B$ we have a groupoid $\M(B)$---i.e., a category all of whose morphisms are isomorphisms---and for every holomorphic map $f \colon B' \rightarrow B$ we have a functor $\M(f)^* \colon \M(B)\rightarrow \M(B')$.
Furthermore, for each $g \colon B'' \rightarrow B'$ we have a natural transformation $\tau_{f, g} \colon \M(g)^*\circ \M(f)^* \Rightarrow \M(f\circ g)^*$.
Finally, the natural transformations $\tau_{f, g}$ are required to be associative whenever we have a chain of three morphisms $B'''\xrightarrow{h} B'' \xrightarrow{g}B' \xrightarrow{f} B$.

For instance, one replaces the functor $M^{/iso}_{1, 1}$ from the previous section with $\M_{1, 1} \colon \man^{op} \rightarrow \mathrm{Groupoids}$ sending $B$ to the groupoid whose objects are elliptic curves over $B$, and morphisms are isomorphisms.

We say that $\M \colon \man^{op} \rightarrow \mathrm{Groupoids}$ is a stack if it satisfies the two conditions below (cf.~\cite[Definition 1.1, Remark 1.2]{hei05}). The pullback of an object $s \in \M(B)$ to an open $B' \subset B$ will be denoted by $s_{|B'}$.
\begin{description}
	\item[Objects glue] take an open cover $(B_i)_{i \in I}$ of a manifold $B$, objects $s_i \in \M(B_i)$ and isomorphisms $\varphi_{ij} \colon s_{i|B_i \cap B_j}\rightarrow s_{j|B_i \cap B_j}$ satisfying the cocycle condition $\varphi_{jk}\circ \varphi_{ij}=\varphi_{ik}$ on triple intersections $B_i \cap B_j \cap B_k$.
	      Then there exists $s \in \M(B)$ together with isomorphisms $\varphi_i \colon s_{|B_i}\rightarrow s_i$ such that $\varphi_{ij}=\varphi_j \circ \varphi_i^{-1}$.
	\item[Isomorphisms glue] take an open cover $(B_i)_{i \in I}$ of a manifold $B$, and two objects $s, s' \in \M(B)$.
	      Then, given isomorphisms $\varphi_i \colon s_{|B_i}\rightarrow s'_{|B_i}$ which agree on double intersections $B_i \cap B_j$, there exists a unique $\varphi \colon s \rightarrow s'$ restricting to $\varphi_i$ on each $B_i$.
\end{description}

A morphism of stacks $F \colon \M \rightarrow \mathsf{N}$ is a ``natural transformation of groupoid-valued functors'': precisely, it consists of the datum of a functor $F_B \colon \M(B) \rightarrow \mathsf{N}(B)$ for each manifold $B$, together with equivalences $\mathsf{N}(f)^*\circ F_B\simeq F_{B'} \circ \M(f)^*$ for each $f \colon B' \rightarrow B$.

Considering every set as a groupoid whose only morphisms are the identity morphisms, sheaves on $\man$ (in particular, objects of $\man$ via the Yoneda embedding) become special examples of stacks.
In this appendix, we will denote the stack associated with a manifold $B$ by $\underline{B}$; in the body of the text we follow the common abuse of denoting the stack associated with a manifold $B$ still by $B$.

\begin{rem}
	\label{2mor}
	There is a notion of morphism between two morphisms of stacks, which we will call transformation of morphisms.

	Namely, let $F_1, F_2 \colon \M \rightarrow \mathsf{N}$ be morphisms of stacks.
	By definition, for each $B$ we have functors $F_{1, B}, F_{2, B} \colon \M(B) \rightarrow \mathsf{N}(B)$. A transformation from $F_1$ to $F_2$ is a collection of natural transformations from  $F_{1, B}$ to $F_{2, B}$ for every complex manifold $B$, compatible with pullbacks (cf.~\cite[Remark 3.1.4]{hei05}; formally, stacks form a 2-category with invertible 2-morphisms).
\end{rem}

\subsubsection{Atlases and analytic stacks}
\label{atlas}

To sum up, the condition defining a stack can be thought of as a ``sheaf condition'' for functors valued in groupoids.
Much as arbitrary sheaves on $\man$ have little geometric significance, we must impose additional conditions on stacks in order to be able to do geometry with them.
This is captured by the crucial notion of atlas, a ``good enough'' approximation of a stack $\M$ by a complex manifold $X$.
More precisely, we want to require the existence of a surjective submersion $X\rightarrow \M$; let us recall the precise meaning of this expression.

Firstly, we recall the definition of fibre product $\M_1 \times_\mathsf{N} \M_2$ of two stacks $F_1 \colon \M_1\rightarrow \mathsf{N}, F_2 \colon \M_2\rightarrow \mathsf{N}$.
For $B \in \man$, objects of $\M_1 \times_\mathsf{N} \M_2(B)$ are triples $(s_1, s_2, \varphi)$ consisting of an object $s_1 \in \M_1(B)$, an object $s_2 \in \M_2(B)$, and an isomorphism $\varphi \colon F_{1, B}(s_1)\rightarrow F_{2, B}(s_2)$.
A morphism from $(s_1, s_2, \varphi)$ to $(s_1', s_2', \varphi')$ is a couple $(\varphi_1, \varphi_2)$ consisting of a morphism $\varphi_1 \colon s_1 \rightarrow s_1'$ and a morphism $\varphi_2 \colon s_2 \rightarrow s_2'$, such that $F_{2, B}(\varphi_2)\circ \varphi=\varphi' \circ F_{1, B}(\varphi_1)$ for every $B$.
We say that a morphism of stacks $F \colon \M \rightarrow \mathsf{N}$ is representable if, for every map of stacks $\underline{B}\rightarrow \mathsf{N}$ from a manifold $B$, the fibre product $\underline{B}\times_\mathsf{N} \M$ is of the form $\underline{B'}$ for a manifold $B'$.

\begin{rem}
	Let us point out that the above notion of representability is very strong: if $\M, \mathsf{N}$ are (stacks attached to) manifolds then we are asking in particular that all the fibres of $F$ are manifolds, which is far from true for a general holomorphic map.

	A better approach to the general theory consists in replacing $\man$ by the category of complex analytic spaces, which admits fibre products (cf.~also~\cite[Remark 3.4]{beno06}).
	However, the morphisms of stacks of interest in this document will satisfy the above condition. Hence, to make the text more accessible, we have chosen to work with the above definition.
\end{rem}

\begin{defin}\label{def:anstack}
	A stack $\M$ is called analytic stack if there is a manifold $X$ and a morphism of stacks $\underline{X}\rightarrow \M$ which is a surjective submersion, i.e., a representable map such that for every manifold $B$ and every morphism $\underline{B}\rightarrow \M$ the projection $\underline{B}\times_{\M}\underline{X}\rightarrow \underline{B}$ is (induced by) a surjective holomorphic submersion of complex manifolds.

	Any manifold $\underline{X} \rightarrow \M$ with the previous property is called an atlas of the stack $\M$.
\end{defin}

\subsubsection{Moduli stacks of curves with marked points}
\label{def-mgma}

Let $g, m$ be two nonnegative integers.
The (2-)functor $\Mgma$ sending a complex manifold $B$ to the groupoid of families of compact Riemann surfaces of genus $g$ over $B$ with $m$ (ordered) disjoint families of marked points is an analytic stack.
If $g \geq 1, m \geq 1$ or $g=0$ and $m \geq 3$, then $\Mgma$ is a quotient stack $[T_{g, m}/\Gamma_{g, m}]$, where $T_{g, m}$ is a Teichm\"{u}ller space and $\Gamma_{g, m}$ is the (pure) mapping class group. The quotient stack $[T_{g, m}/\Gamma_{g, m}]$ is the stack attached to the groupoid $\Gamma_{g, m} \times T_{g, m} \rightrightarrows T_{g, m}$ whose maps are the projection on $T_{g, m}$ and the $\Gamma_{g, m}$-action; see \cite[Definition 3.3]{noo05}, and \cite[Theorem 12.6]{noo05} or \cite[Example 1.5]{hei05} for the interpretation of quotient stacks in terms of torsors.

The above description of $\Mgma$ as a quotient stack rests on the fact that $T_{g, m}$ is a moduli space of Riemann surfaces with marked points and ``Teichm\"{u}ller structure''; cf.~\cite[Th\'eor\`eme 3.1]{gro62}, surveyed in \cite{acjp16}, for the case $m=0$, and \cite[Theorem 3.6, Proposition 3.9]{eng75} or \cite{ear73} for the general case.

\subsection{The fundamental group of an analytic stack}

In this section we give the definition of the (topological) fundamental group of an analytic stack, borrowed from~\cite{noo05}.
This is done in two steps: firstly, one defines the topological stack ``underlying'' an analytic stack; secondly, one extends the usual definition of fundamental group of a topological space to topological stacks.
The first step rests on the dictionary between analytic stacks and (complex analytic) groupoids, which gives a concrete way to represent stacks in terms of ``manifolds generators and relations''.

\subsubsection{Presentation of analytic stacks via groupoids}

Let $\M$ be an analytic stack and $\underline{X}\rightarrow \M$ an atlas.
The fibre product $\underline{X}\times_{\M} \underline{X}$ is of the form $\underline{R}$ for a manifold $R$ endowed with two projections to $X$, which define the source and target morphism of a groupoid $X_\bullet$ (in $\man$).
Conversely, each groupoid $R\rightrightarrows X$ gives rise to a quotient stack $[X/R]$~\cite[\S~3.2]{noo05},~\cite[\S~3]{hei05}.
Heuristically, the groupoid $X_\bullet$ can be thought of as a ``presentation'' of the stack $\M$: the two projection maps describe which points of $X$ should be identified to obtain $\M$.
However, it can happen that a point $x \in X$ gets identified with itself via more than one $r \in R$; this information is remembered by the stack $\M$---and is related precisely to non-trivial automorphisms of the objects parametrised by $\M$---but it would be lost looking at the quotient topological space $X/R$.

\subsubsection{Topological stack associated with an analytic stack}
\label{topst}

Let $\M$ be an analytic stack; choose an atlas $\underline{X} \rightarrow \M$ and let $X_\bullet=R \rightrightarrows X$ be the resulting groupoid.
Consider the topological spaces $X^{top}$ and $R^{top}$ underlying $X$ and $R$, and the groupoid in topological spaces $X_{\bullet}^{top}=R^{top} \rightrightarrows X^{top}$; the associated stack on the category of topological spaces is denoted by $\M^{top}$, and is called the topological stack attached to $\M$ (note that $\M^{top}$ is a topological stack in the sense of~\cite[Definition 13.8]{noo05}, using as local fibrations those coming from~\cite[Example 13.1.3]{noo05}).
The fact that $\M^{top}$ is well defined, i.e., it does not depend on the choice of the atlas, follows from the next lemma.

\begin{lem}
	Let $\M$ be an analytic stack and let $\underline{X}_1\rightarrow \M$, $\underline{X}_2\rightarrow \M$ be two atlases.
	The quotient stacks $[X^{top}_{1}/R^{top}_1]$ and $[X^{top}_{2}/R^{top}_2]$ are equivalent.
\end{lem}

\begin{proof}
	The lemma can be proved looking at an atlas refining $\underline X_1$, $\underline X_2$, as sketched in~\cite[p. 79]{noo05}.
	For completeness, let us give the details.
	The fibre product $\underline{X}_1\times_\M \underline{X}_2$ is of the form $\underline{X}_3$ for a manifold $X_3$.
	The projection $X_3\rightarrow X_1$ (resp. $X_3\rightarrow X_2$) is a (locally trivial) $X_{2, \bullet}$-bundle (resp. $X_{1, \bullet}$-bundle) in the sense of~\cite[\S~3]{hei05}, and the two actions defining the bundle structures commute.
	Hence, the same properties are true for the maps $X_3^{top}\rightarrow X_i^{top}$ for $i=1, 2$, and the lemma follows from~\cite[Lemma 3.2]{hei05}.
\end{proof}

\subsubsection{The fundamental group}

A pointed analytic stack is a couple $(\M, x)$ consisting of an analytic stack $\M$ and a map $\underline{*} \rightarrow \M$, corresponding to an object $x \in \M(*)$ (the image of the identity morphism of $*$).
We denote by $(\M^{top}, x)$ the associated pointed topological stack. Any pointed topological space $(T, x)$ gives rise to a pointed topological stack $(\underline{T}, x)$.
The following definition (coming from~\cite[Definition 17.5]{noo05}) is analogous to the usual definition of the fundamental group of a topological space; the subtlety and interest of the definition stem from the fact that the expressions ``homotopy'' and ``map of pairs'' have a more refined meaning in the world of stacks than in the context of topological spaces.
This will be recalled below.

\begin{defin}
	\label{fundgp-stack}

	The fundamental group of a pointed analytic stack $(\M, x)$, denoted by $\pi_1(\M, x)$, is the set of homotopy classes of maps of pairs $(\underline{S}^1, 1)\rightarrow (\M^{top}, x)$, endowed with the group structure defined in~\cite[pp. 59, 60]{noo05}.
\end{defin}

\begin{rem}
	Let us give some comments on the previous definition.
	The main point is that the notions of map of pairs and homotopy between continuous maps of topological spaces can be expressed via the commutativity of certain diagrams.
	However, as explained in Remark~\ref{2mor}, in the world of stacks we have a notion of transformation between morphisms; to extend the classical definitions to the context of stacks, one includes a transformation from a map $F_1$ to a map $F_2$ whenever usually one would have asked the two maps to be equal.

	For instance, a map $(\underline{S}^1, 1)\rightarrow (\M^{top}, x)$ consists by definition~\cite[Definition 17.1]{noo05} of the datum of a morphism $F \colon \underline{S}^1\rightarrow \M^{top}$ together with a transformation between the composition $\underline{*}\xrightarrow{1} \underline{S}^1 \rightarrow \M^{top}$ and the map $\underline{*}\rightarrow \M^{top}$ corresponding to $x$.
	Such a transformation is induced by an isomorphism between $F(1)$ and $x$.
	In other words, we do not ask $F$ to map $1$ to $x$, but only to an object isomorphic to $x$, and we incorporate isomorphisms $F(1)\simeq x$ into our definition of ``loop''.
	In particular, we see that automorphisms of objects of $\M(*)$ contribute to the fundamental group of the stack, in a way explained more precisely in~\cite[\S~17]{noo05}.
	Finally, the notion of homotopy between maps of pairs is defined in~\cite[Definition 17.2]{noo05}.
\end{rem}

\begin{rem}\leavevmode
	\begin{enumerate}
		\item If $\M = \underline X$ for a manifold $X$, then the transformations mentioned in the previous remark are always trivial, hence Definition~\ref{fundgp-stack} recovers the usual fundamental group of $(X, x)$.

		\item There is an alternative way to define $\pi_1(\M^{top}, x)$ only involving fundamental groups of topological spaces: choose an atlas $\underline X\rightarrow \M$ and consider the topological groupoid $X_\bullet^{top}=R^{top} \rightrightarrows X^{top}$.
		      Consider the classifying space $BX_\bullet$ of $X_\bullet$, as defined in~\cite[\S~4.1]{noo12}.
		      It comes with a natural map $BX_{\bullet}\rightarrow \M$~\cite[Proposition 6.1]{noo12}; one defines $\pi_1(\M, x)=\pi_1(BX_\bullet, \tilde{x})$ for a point $\tilde{x}$ above $x$. The fact that this definition is independent of choices and coincides with the one given above follows from~\cite[proofs of Theorems 10.5--6.3]{noo12}.
	\end{enumerate}
\end{rem}

\subsubsection{Example}
\label{fundgp-quot}

The approach described in the previous remark allows us to describe fundamental groups of quotient stacks, cf.~\cite[Remark 4.3]{noo12}.
Firstly, if $\M=[*/G]$ for a (discrete) group $G$ then $X=*$ and $BX_\bullet=EG/G$, where $EG$ is the universal principal $G$-bundle (in particular, it is a simply connected space on which $G$ acts freely), so $\pi_1(\M, *)=G$.
More generally, if $G$ acts on a topological space $X$ and $\M$ is the quotient stack $[X/G]$ then $BX_\bullet$ is the quotient $X\times_G EG$ of $X\times EG$ by the antidiagonal action of $G$ (Borel construction).

In particular, if $X$ is simply connected then $\pi_1([X/G], x)=G$ (regardless of the stabilisers of points of $X$).
For instance, with the notation of~\cref{def-mgma}, the fundamental group of $\Mgma$ is $\Gamma_{g, m}$.

\bibliographystyle{alphaabbr}
\bibliography{wild}

\end{document}